\UseRawInputEncoding

\documentclass[hyperref]{article}

\usepackage{amsfonts}
\usepackage{subcaption}
\usepackage{color}
\usepackage{graphicx}
\usepackage{float}
\usepackage{caption}
\usepackage{amssymb}
\usepackage{amsmath}
\allowdisplaybreaks[4]
\usepackage{mathtools}
\usepackage[thmmarks,amsmath]{ntheorem}
\usepackage{bm}
\usepackage[colorlinks,linkcolor=cyan,citecolor=cyan]{hyperref}
\usepackage{extarrows}
\usepackage{mathrsfs}
\usepackage[font=footnotesize,skip=0pt,textfont=rm,labelfont=rm]{caption}
\usepackage{booktabs}
\usepackage{appendix}
\usepackage[subfigure]{tocloft}
\usepackage{theorem}
\usepackage[a4paper,left=2.5cm,right=2.5cm,top=2.5cm,bottom=2.5cm]{geometry}
\usepackage{multirow}
\usepackage{url}
\usepackage{diagbox}
\usepackage{cleveref}
\usepackage[round,sort]{natbib}
\setcitestyle{authoryear,round}
\usepackage[linesnumbered,boxed,ruled,commentsnumbered]{algorithm2e}
\usepackage{comment}

{
	\theoremstyle{nonumberplain}
	\theoremheaderfont{\bfseries}
	\theorembodyfont{\normalfont}
	\theoremsymbol{\mbox{$\Box$}}
	\newtheorem{proof}{Proof}
}
\newtheorem{theorem}{Theorem}[section]
\newtheorem{lemma}{Lemma}[section]
\newtheorem{definition}{Definition}[section]
\newtheorem{assumption}{Assumption}[section]

\newtheorem{corollary}{Corollary}[section]
\newtheorem{proposition}{Proposition}[section]
{
	\theoremheaderfont{\bfseries}
	\theorembodyfont{\normalfont}
	\newtheorem{remark}{Remark}[section]
}

\graphicspath{{figure/}}                                 
\usepackage{times}  

\begin{document}
	
	\title{\textbf{On the Maximum and Minimum of a Multivariate Poisson Distribution}}
	\date{18th September 2025}
	\author{ Zheng Liu$^1$, Feifan Shi$^{1}$, Jing Yao$^{1}$\thanks{
			Corresponding author, e-mail: j.yao@suda.edu.cn}, Yang Yang$^{1}$ \\
		{{\small $^1$ Center for Financial Engineering, Soochow University, China }}}
	\maketitle
	
	\begin{abstract}
		In this paper, we investigate the cumulative distribution functions (CDFs) of
		the maximum and minimum of multivariate Poisson distributions with three dependence structures, namely, the common shock, comonotonic shock and thinning-dependence models. In particular, we formulate the definition of a thinning-dependent multivariate Poisson distribution based on \cite{wang2005correlated}. We derive explicit CDFs of the maximum and minimum of the multivariate Poisson random vectors and conduct asymptotic analyses on them. Our results reveal the substantial difference between the three dependence structures for multivariate Poisson distribution and may suggest an alternative method for studying the dependence for other multivariate distributions. We further provide numerical examples demonstrating obtained results. 
	\end{abstract}
		\section{Introduction}
	
	Studies on the distribution of the maximum (minimum) of a $d$-dimension random vector $\bm{X}=\left( X_{1},\cdots ,X_{d}\right) $ have been garnering considerable ongoing attention in the literature. Such a distribution of the maximum (minimum) serves as a tailor-made tool for modeling extreme events in diverse fields. For instance, it is widely employed in hydrology \citep{salas2018techniques}, biostatistics \citep{cheng2008statistical}, as well as insurance and finance applications \citep{embrechts2013modelling}. Therefore, the investigation of the distribution of $M^{\bm{X}}=\max \left( 
	X_{1},\cdots ,X_{d}\right) $ or equivalently that of $m^{\bm{X}}=\min \left(
	X_{1},\cdots ,X_{d}\right) $ holds particular importance and many scholars 
	have delved into this topic in the literature.
	
	For the classic normal distribution, in the case of $d=2$ (i.e., the bivariate normal distribution),  \cite{basu1978identifiability} investigated the distribution of $M^{\bm{X}}$, while \cite{cain1994moment} obtained its first two moments using the moment-generating function. Building upon this work,  \cite{ker2001maximum} examined how the mean and variance of $M^{\bm{X}}$ behave with changes in the mean, variance, and covariance of $\bm{X}$. When considering $d\geq 3$ (i.e., the multivariate normal distribution case), \cite{nadarajah2019distribution} derived explicit expressions for the probability density function (PDF) and moments of $M^{\bm{X}}$. Furthermore, as a variant of the normal distribution, the log-normal distribution is frequently utilized for modeling stochastic price/loss in finance and insurance. Its maximum and minimum hold considerable significance in applications, and these distributions have been studied by \cite{lien1986moments, lien2005minimum}.
	
	For other distributions,  \cite{jamalizadeh2010distributions} investigated the distributions of order statistics and linear combinations of order statistics derived from an elliptical distribution. They showed that these distributions can be expressed as mixtures of unified skew-elliptical distributions. \cite{hakamipour2011extremes} derived the distributional properties of the minimum and maximum when $\bm{X}$ has a bivariate Pareto distribution, and this result was extended to the multivariate case by \cite{nadarajah2013maximum}. Unlike previous studies, \cite{arellano2008exact} derived a more general expression for the maximum distribution applicable to absolutely continuous random vectors, and they specifically investigated the scenario where the random vector follows an elliptically contoured distribution.
	
	However, as mentioned above, the existing literature focuses on continuous random vectors, leaving a research gap concerning multivariate discrete distributions. The primary reasons for this gap could be two-fold. On one hand, discrete random vectors do not have PDF, rendering many abovementioned methods, such as \cite{arellano2008exact}'s general expression, inapplicable. For discrete distributions, we are limited to study the cumulative distribution function (CDF) and probability mass function (PMF). On the other hand, the probability that multiple variables within a random vector reach maximum or minimum values simultaneously is strictly larger than 0, i.e.  
	\begin{equation*}
		\mathbb{P}\Bigl(\exists\, I \subseteq \{1,\dots,d\}, \, |I|\geq 2: \; 
		M^{\bm{X}} = X_i \;\; \forall i\in I\Bigr) > 0,~ \left( \text{or}\
		\ \mathbb{P}\Bigl(\exists\, I \subseteq \{1,\dots,d\}, \, |I|\geq 2: \; 
		m^{\bm{X}} = X_i \;\; \forall i\in I\Bigr) > 0\right) ,
	\end{equation*}
	which makes studying the maximum (minimum) of multivariate discrete random vectors even more challenging, especially for vectors exhibiting dependence.
	In the continuous case, these probabilities are all equal to zero.
	
	Despite these challenges, addressing this research gap is crucial due to the widespread applications and significance of discrete distributions in various fields. Among discrete distributions, the Poisson distribution, as a fundamental cornerstone in probability theory, plays a pivotal role in mathematical statistics, financial risk management, and biology \citep[see][]{lindskog2003common,chua2005poisson,inouye2017review,bailey1995statistical}. In this regard, we focus on the distribution of maximum and minimum in the context of multivariate Poisson distributions in this paper. Particularly, the common shock structure \citep{yuen2015optimal} is one of the most prevalent and effectively depicts the dependence on many real-world events. For instance, within the insurance industry, the claims of different types of insurance business are often dependent. A natural disaster may trigger various types of insurance claims, including but not limited to medical, death, and property claims, etc. Such systematic impact of the natural disaster can be captured by a common shock in generating multivariate claims. Furthermore,  \cite{schulz2021multivariate} pointed out that the common shock structure inherently imposes constraints on pairwise covariances, requiring them to be positive and identical, which may not accurately represent some real-world scenarios. To overcome this limitation,  \cite{schulz2021multivariate} extended the common shock structure to a comonotonic shock structure. This extension provides enhanced flexibility in capturing the underlying dependence structure. By contrast, the thinning-dependence structure, introduced by \cite{wang2005correlated} to characterize continuous-time risk models in credit risk theory and the insurance business, incorporates a number of common background risks that may cause common claims of $X_{i},i=1,...,d$. Likewise, in the insurance industry example mentioned earlier, the thinning-dependence structure not only captures the impact of background common risk sources (e.g. a natural disaster) but also gives the probabilities to different types of insurance claims ($X_{i}$) triggered by the same background risk sources, offering another flexible dependence with a more interpretive structure.
	
	Although the three dependence structures share a similar idea by correlating marginal Poisson random variables via certain common factors, they distinguish from each other in how the common factors contribute to the marginal random variables $X_{i},i=1,\dots ,d$. The common shock structure utilizes a unique random variable $Y_{0}$ to depict the impact of the common claim. This seems intuitive and straightforward, but also imposes specific constraints. The comonotonic shock is a natural extension to common shock. It replaces the unique random shock variable $Y_{0}$ in marginal 
	random variables by a comonotonic random vector\footnote{
		A random vector $(Z_{1},\dots ,Z_{d})$ is comonotonic if and only if each 
		component can be represented in terms of a common underlying uniform random 
		variable on the unit interval. That is, for a uniform random variable $U\sim
		\mathcal{U}(0,1)$, the comonotonic vector can be written as  
		\begin{equation*}
			\left( Z_{1},\dots ,Z_{d}\right) =\left( F_{1}^{-1}(U),\dots
			,F_{d}^{-1}(U)\right) ,
		\end{equation*}
		where for each $j\in \{1,\dots ,d\}$, $F_{j}$ denotes the distribution of $
		Z_{j}$. More details can be found in \cite{schulz2019multivariate}. 
		Moreover, in this paper, for a real function $F$, $F^{-1}(u)=\inf \{t\in  
		\mathbb{R}:F(t)\geq u\}$.} $\bm{Z}=(Z_{1},\dots ,Z_{d})$, providing more flexible \textquotedblleft shocks\textquotedblright\ in the dependent structures. Nevertheless, both common shock and comonotonic shock structures incorporate the common factor as \textquotedblleft shocks\textquotedblright\ \textit{directly} into the system. Specifically, they add the \textquotedblleft shocks\textquotedblright\ (the shock variable $Y_{0}$ or the shock vector $\bm{Z}$) into another independent random vector $\bm{Y}=(Y_{1},\dots ,Y_{d})$, which stands for the idiosyncratic marginal shocks.   
	On the contrary, simultaneous claims (shocks) in the thinning-dependence structure are triggered by commonly shared background random variables (risk sources), tying marginal Poisson random variables up via an explanatory probabilistic mechanism. We present more details on the three dependence structures in \Cref{sec2}.
	
	In this paper, we explore the maximum and minimum of the multivariate
	Poisson distribution with dependence. Our contributions to the existing
	literature can be summarized in threefold. Firstly, to the best of our
	knowledge, we are the first to investigate the distribution of maximum and
	minimum of the multivariate Poisson distribution with dependence. More
	specifically, we explicitly derive the CDFs of the $M^{\bm{X}}$ and $m^{%
		\bm{X}}$ of the multivariate Poisson distribution under the three
	aforementioned dependent structures, namely, the common shock, comonotonic
	shock, and thinning-dependence structures. In doing so, we theoretically
	formulate a multivariate Poisson distribution with thinning-dependence
	structure, which broadens the scope of this dependence structure from
	stochastic processes to multivariate probability distributions, offering a pedagogical contribution. 
	
	Secondly, we perform asymptotic analysis of all derived CDFs and obtain several interesting results that underpin the substantial differences between the three dependence structures. For example, when the dimension $d\rightarrow \infty $, we find that the asymptotic behaviors of derived CDFs under the common shock and comonotonic shock structures are similar, whereas the one under the thinning-dependence structure exhibits evident discrepancy. Specifically, the asymptotic behavior of $M^{\bm{X}}$ aligns with the commonly shared background risk sources (more precisely, the convolution of background random variables) under the thinning-dependence structure. By contrast, under the common shock and comonotonic shock structures, the CDFs of $M^{\bm{X}}$ are asymptotically determined by the idiosyncratic random vector $\bm{Y}$ ($Y_{0}$ for common shock structure), which is independent of common shocks. Furthermore, despite the tight connection, the asymptotic analysis on $d$ for comonotonic structure is mathematically much more challenging than common shock structure because the dimension of the shock vector $\bm{Z}$ increases along $d$, indicating the number of comonotonic shock variables $Z_{i}$ also approach infinity. As such, the comonotonic structure differs significantly from the common shock structure and requires additional assumptions on the comonotonic shock vector $\bm{Z}$ to obtain similar asymptotic results. 
	Such finding reveals the intrinsic difference between the common shock and comonotonic shock structures, and the difference is more obvious as $d$ is high. Moreover, we also consider the asymptotic scenarios when certain key parameters dominate the dependence structures. For instance, it is intuitive to expect that when the rate of the common shock variable $Y_{0}$ dominates the other Poisson rates, the distribution of $M^{\bm{X}}$ and $m^{\bm{X}}$ should considerably depend on the distribution of $Y_{0}$. Our results validate this analytically and quantify the convergency rate between the CDF of $Y_{0}$ and those of $M^{\bm{X}}$ and $m^{\bm{X}}$, which is in line with intuitive expectations. Overall, our asymptotic analyses, to some extent, provide an alternative new way to study similar dependence structures for other multivariate distributions from an asymptotic standpoint.
	
	
	Thirdly, we present numerical examples to demonstrate the influence of different dependent structures on the $M^{\bm{X}}$ and $m^{\bm{X}}$ of multivariate Poisson distributions. To highlight this impact, it is essential to ensure consistent settings for a relatively \textquotedblleft  fair\textquotedblright\ comparison of the three structures. 
	To this end, we first consider the case where all marginal Poisson rates are equal. Regarding the dependence structure, requiring all pairwise correlations to be identical is highly restrictive, especially when $d$ is large, as it would enforce equality across $(d\times d-1)/2$ correlation coefficients. Therefore, instead of pairwise correlation, we focus on the average correlation coefficient, which serves as a tractable and independent measure of the overall dependence in a multivariate random vector.
	Based on the above settings, we compare their CDFs in numerical examples under two scenarios: one where the common shock and thinning-dependence structures differ significantly and another where they are very similar. Our results show that when the structures are similar, the CDF values are close. However, when the structural differences are significant, the difference in CDF is pronounced.  These findings strongly demonstrate that the construction of dependent structures plays a crucial role in determining the maximum and minimum values. Compared with \cite{karlis2003algorithm,schulz2021multivariate} and \cite{wang2005correlated}, our numerical results further analyze these three dependent structures from the perspective of their maximum (minimum) values, offering an additional viewpoint to discuss the similarities and differences between them.
	
	The rest of paper is organized as follows. \Cref{sec2} outlines the definition of the multivariate Poisson-distributed random variable under common shock, comonotonic shock, and thinning-dependence structure. The CDFs of $M^{\bm{X}} $ and $m^{\bm{X}}$ are obtained in \Cref{sec3}. In \Cref{sec4}, we show the results of our asymptotic analysis under the three structures. \Cref{sec5} presents the numerical illustration and \Cref{sec6} concludes the paper with the discussions.
	
	\section{Multivariate Poisson distribution with dependence} \label{sec2}
	
	In this section, we utilize three distinct dependence structures: common shock, comonotonic shock, and thinning-dependence structures, to construct multivariate Poisson distributions. Specifically, we provide the mathematical definitions of these three multivariate Poisson distributions. Throughout this paper, $\mathcal{P}(\theta )$ denotes a univariate Poisson distribution with rate $\theta $, $G_{\theta }(\cdot )$ represents its CDF, and $\bar{G}_{\theta }(\cdot )$ is the corresponding survival function.
	
	\subsection{Common shock structure}
	
	In the literature, several methods for constructing the common shock structure have been proposed, such as those by \cite{mahamunulu1967note} and  \cite{kawamura1976structure}, which consider multiple random variables as the common shock components. Although these approaches are intuitive and flexible, deriving the explicit expression of the distribution becomes complex, as noted by \cite{karlis2007finite} and \cite{inouye2017review}. This complexity poses challenges for studying the distribution of the maximum or minimum. Therefore, in this paper, following the approach of \cite{karlis2003algorithm}, we consider only the single common shock variable case. The definition of the multivariate common shock Poisson distribution is given as follows:
	
	\begin{definition}
		\label{def-common-shock} A random vector $\left( X_{1},\dots ,X_{d}\right) $
		is said to have a multivariate common shock Poisson distribution with 
		parameters $\left( \theta _{0},\theta _{1},\cdots ,\theta _{d}\right) \in 
		(0,\infty )^{d+1}$ if it can be expressed as follow:  
		\begin{equation}
			X_{i}=Y_{0}+Y_{i},\quad i=1,2,\dots ,d,  \label{common shock def}
		\end{equation}
		where $Y_{i}\sim \mathcal{P}(\theta _{i})$, $i=0,1,\dots ,d$ are mutually 
		independent Poisson random variables. 
	\end{definition}
	
	The multivariate Poisson model with the common shock structure is applied in
	various fields, including market research and epidemiology  \citep[see][]%
	{karlis2003algorithm}. However, this structure also imposes strong 
	limitations. Specifically, it imposes constraints on the correlation between
	any two variables in a multivariate common shock Poisson vector. For 
	arbitrary pair of Poisson random variables with rates $\theta _{i}$ and $
	\theta _{j}$, \cite{griffiths1979aspects} demonstrated that
	their correlation falls in the interval $[\rho _{\min } ,\rho _{\max }] $ where  
	\begin{align*}
		\rho _{\min }& =\frac{1}{\sqrt{\theta _{i}\theta _{j}}}\left[ -\theta
		_{i}\theta _{j}-\sum_{m\in \mathbb{N}}^{{}}\sum_{n\in \mathbb{N}}^{{}}\min
		\left\{ 0,G_{\theta _{i}}\left( m\right) +G_{\theta _{j}}\left( n\right)
		-1\right\} \right] , \\
		\rho _{\max }& =\frac{1}{\sqrt{\theta _{i}\theta _{j}}}\left[ -\theta
		_{i}\theta _{j}+\sum_{m\in \mathbb{N}}^{{}}\sum_{n\in \mathbb{N}}^{{}}\min
		\left\{ \overline{G}_{\theta _{i}}\left( m\right) ,\overline{G}_{\theta
			_{j}}\left( n\right) \right\} \right] .
	\end{align*}
	Nevertheless, within the common shock structure, the correlation coefficient
	will not reach $\rho _{\max }$ unless the marginal Poisson rates are 
	identical.
	
	\subsection{Comonotonic shock structure}
	
	To establish a more flexible dependence structure, \cite{schulz2021multivariate} substituted a comonotonic vector $\bm{Z}$ for the single common shock variable $Y_0$, allowing each $X_i$, $i=1,\dots,d$, to have its own comonotonic shock variable $Z_i$. It is important to highlight that since $\bm{Z}$ is comonotonic, all comonotonic shock variables $Z_i$ are controlled by the same uniform random variable $U$, essentially reflecting the influence of the common factor. This structure enhances the flexibility of dependent structures while maintaining intuitiveness and interpretability. Its mathematical definition is as follows:
	
	\begin{definition}
		\label{def-comonotonic shock} A random vector $\left( X_{1},\dots 
		,X_{d}\right) $ is said to follow the multivariate comonotonic shock Poisson
		distribution with parameters $\Lambda _{d}=\left( \lambda _{1},\dots 
		,\lambda _{d}\right) \in \left( 0,\infty \right) ^{d}$ and $\theta $$\in $$ %
		\left[ 0,1\right] $, if it can be expressed in  
		\begin{equation*}
			\left( X_{1},\dots ,X_{d}\right) =\left( Y_{1},\dots ,Y_{d}\right) +\left(
			Z_{1},\dots ,Z_{d}\right) ,
		\end{equation*}
		in terms of two independent random vectors $\left( Y_{1},\dots ,Y_{d}\right)
		$ and $\left( Z_{1},\dots ,Z_{d}\right) $ such that\newline
		(i) $Y_{1},\dots ,Y_{d}$ are mutually independent and, for all $j\in \left\{
		1,\dots ,d\right\} ,Y_{j}\sim \mathcal{P}\left( \left( 1-\theta \right) 
		\lambda _{j}\right) $;\newline
		(ii) $Z_{1},\dots ,Z_{d}$ are comonotonic and, for all $j\in \left\{ 1,\dots
		,d\right\} ,Z_{j}\sim \mathcal{P}\left( \theta \lambda _{j}\right) $. 
	\end{definition}
	
	Notably, the marginals are mutually independent when $\theta =0$, and $\left( X_{1},\dots ,X_{d}\right) $ is comonotonic when $\theta =1$. Moreover, when $\lambda _{1}=\lambda _{2}=\cdots =\lambda _{d}$, the multivariate comonotonic shock Poisson distribution reduces to the multivariate common shock Poisson distribution. In contrast to the common shock structure, the introduction of the comonotonic shock vector enhances the flexibility of this dependent structure. For instance, the correlation coefficient of any pair $(X_{i},X_{j})$, $i\neq j$, can attain $\rho _{\max} $ without requiring the marginal Poisson rates to be equal.
	
	\subsection{Thinning-dependence structure}
	
	Since \cite{wang2005correlated} introduced correlated Poisson processes with a thinning-dependence structure, most studies have applied this framework exclusively to describe the correlation of stochastic processes. In this paper, we extend the thinning-dependence structure to random vectors and propose a multivariate thinning-dependence Poisson distribution, defined as follows.
	
	\begin{definition}
		\label{def-thinning-dependence} Assuming  
		\begin{equation}
			\bm{X}=(X_{1},\dots ,X_{d}),\quad X_{j}=\sum_{k=1}^{l}X_{j}^{k},\quad
			j=1,\dots ,d,  \label{thinning def eq}
		\end{equation}
		and $Y_{1},Y_{2},\dots ,Y_{l}$ are independent Poisson random variables with
		parameters $\theta _{1},\dots ,\theta _{l}$, respectively. Then, $\bm{X}$ is said to follow the multivariate 
		thinning-dependence Poisson distribution, if, \newline
		(i) the random variables $X_{j}^{k}|Y_{k}=y_{k},j=1,\dots ,d,k=1,\dots ,l$ 
		follow the binomial distribution with parameters $y_{k}$ and $p_{j}^{k}$, 
		that is,  
		\begin{equation}
			\mathbb{P}\left( X_{j}^{k}=z|Y_{k}=y_{k}\right) =\binom{y_{k}}{z}
			(p_{j}^{k})^{z}\left( 1-p_{j}^{k}\right) ^{y_{k}-z},  \label{def-thinning eq}
		\end{equation}
		where $z\leq y_{k}$; \newline
		(ii) $X_{1}^{k},\dots ,X_{d}^{k}$ are conditionally independent given $
		Y_{k}=y_{k}$, $k=1,\dots ,l$; \newline
		(iii) the $l$ random vectors $(Y_{1},X_{1}^{1},\dots ,X_{d}^{1}),\dots 
		,(Y_{l},X_{1}^{l},\dots ,X_{d}^{l})$ are independent. 
	\end{definition}
	
	\begin{remark}
		\label{remark 2.1} For $j=1,\dots,d$, $k=1,\dots,l$, the random variables $
		X_j$ and $X_j^k$ are all Poisson-distributed. Specifically, we have $
		X_j^k\sim\mathcal{P}(p_j^k\theta_k)$ and $X_j\sim\mathcal{P}
		(\sum_{k=1}^{l}p_j^k \theta_k)$, for $j=1,\dots,d$, $k=1,\dots,l$. This 
		follows from Definition \ref{def-thinning-dependence}, where  
		\begin{equation*}
			\begin{split}
				\mathbb{P}\left(X_j^k=z\right) = \sum_{y_k=z}^{\infty}\mathbb{P}
				\left(X_j^k=z | Y_k=y_k\right) \mathbb{P}(Y_k=y_k)
				=e^{-(p_{j}^{k}\theta_{k})} \frac{(p_j^k\theta_k)^z}{z!},
			\end{split}%
		\end{equation*}
		implying $X_j^k\sim\mathcal{P}(p_j^k\theta_k)$. Moreover, due to the 
		independence of $X_j^1,\dots, X_j^l$ mentioned in (iii) of Definition \ref%
		{def-thinning-dependence}, we have $X_j\sim\mathcal{P}(\sum_{k=1}^{l}p_j^k 
		\theta_k)$. 
	\end{remark}
	
	\begin{remark}
		\label{remark2.2-1}We provide an explanation of credit risk to Definition \ref{def-thinning-dependence} to further illustrate the idea behind the dependence structure. Consider $d$ firms with a solo systematic background risk (e.g., the default risk of major clients who have business with all firms). Let $Y_{1}$ and $X_{j}^{1},j=1,\dots ,d,$ denote the default times of a common client and the firms, respectively. For each firm, the systematic default of $Y_{1}$ may trigger the $j$-th firm's default with probability $p_{j}^{1}$. Hence, it is straightforward to see $X_{j}^{1}|Y_{1}=y_{1}$ indicates $y_{1}$ trials of systematic default event, which follows a binomial distribution, and these trials are mutually independent when the solo background risk is known.
	\end{remark}
	
	\begin{remark}
		\label{remark2.2} The thinning-dependence structure may reduce to common
		shock structure. Specifically, consider the case where $l=d+1$, and 
		\begin{equation*}
			\begin{bmatrix}
				p_{1}^{1} & p_{1}^{2} & \cdots  & p_{1}^{d} & p_{1}^{d+1} \\ 
				p_{2}^{1} & p_{2}^{2} & \cdots  & p_{2}^{d} & p_{2}^{d+1} \\ 
				\vdots  & \vdots  & \ddots  & \vdots  & \vdots  \\ 
				p_{d}^{1} & p_{d}^{2} & \cdots  & p_{d}^{d} & p_{d}^{d+1}%
			\end{bmatrix}%
			=%
			\begin{bmatrix}
				1 & 0 & \cdots  & 0 & 1 \\ 
				0 & 1 & \cdots  & 0 & 1 \\ 
				\vdots  & \vdots  & \ddots  & \vdots  & \vdots  \\ 
				0 & 0 & \cdots  & 1 & 1%
			\end{bmatrix}%
			,
		\end{equation*}%
		where $p_{j}^{k}$, $j=1,\dots ,d$, $k=1,\dots ,d+1$ are the probabilities given in Definition \ref{def-thinning-dependence}. In this scenario, background random variable $Y^{d+1}\sim \mathcal{P}(\theta_{d+1})$ affects $X_{j}^{d+1}$, $j=1,\dots ,d$ with probability 1, implying that the $Y^{d+1}$ reduces to the common shock variable. Meanwhile, for $k=1,\dots ,d$, the background random variable $Y_{k}\sim \mathcal{P}(\theta _{k})$ affects only $X_{k}^{k}$ with probability 1 and does not affect $X_{j}^{k}$ , for $j\neq k$, implying $X_{k}^{k}=Y^{k}$, $X_{j}^{k}=0$. Thus, according on Definition \ref{def-thinning-dependence}, for $j=1,\dots ,d$, we have $X_{j}=\sum_{k=1}^{d+1}X_{j}^{k}=Y_{j}+Y_{d+1}$, indicating that the multivariate thinning-dependence Poisson distribution reduces to the multivariate common shock Poisson distribution under this parameter setting. 
	\end{remark}
	
	\section{CDFs of maximum and minimum of multivariate Poisson distribution}\label{sec3}
	
	In this section, we explicitly derive CDFs for the maximum and minimum of the multivariate Poisson distribution under the common shock, comonotonic shock, and thinning-dependence structures.
	
	Before presenting the aforementioned results, we first introduce the CDF 
	\footnotemark, and survival function of the univariate Poisson random 
	variable in the following lemma. \footnotetext[2]{
		Notably, the distributions studied in this paper are discrete, all defined 
		on non-negative integers. Let $\hat{F}(x)$ denote the CDF when $x \in\mathbb{%
			\ N}$. When $x>0$ and $\notin\mathbb{N}$, $F(x)=\hat{F}(\lfloor x\rfloor)$, 
		where $\lfloor \cdot \rfloor$ represents rounding down. Therefore, for 
		convenience, we only need to focus on the case where the variable $x$ 
		assumes non-negative integer values in Sections 3 and 4 unless otherwise 
		specified.}
	
	\begin{lemma}
		\label{poisson lemma} Let a random variable $X\sim\mathcal{P}(\lambda)$, 
		then the CDF and survival function of $X$ are
		\begin{equation*}
			G_\lambda(x) = \frac{\Gamma(x+1,\lambda)}{\Gamma(x+1)},\quad
			\bar{G}_\lambda(x) = \frac{\gamma(x+1,\lambda)}{\Gamma(x+1)},
		\end{equation*}
		respectively, where  
		\begin{equation*}
			\Gamma(z)= \int_{0}^{\infty}t^{z-1}e^{-t}dt,\quad \Gamma(z,x)=
			\int_{x}^{\infty}t^{z-1}e^{-t}dt,\quad \gamma(z,x)=
			\int_{0}^{x}t^{z-1}e^{-t}dt,
		\end{equation*}
		are the gamma function, upper incomplete gamma function, and lower 
		incomplete gamma function, respectively. 
	\end{lemma}
	
	\subsection{Common shock structure}
	
	For the multivariate common shock Poisson distribution defined in Definition \ref{def-common-shock}, we can obtain the CDFs of the maximum and minimum by the conditional technique in the following theorem.
	\begin{theorem}
		Suppose $\bm{X}=(X_{1},\dots ,X_{d})$ follows the multivariate common shock 
		Poisson distribution defined in Definition \ref{def-common-shock}, then the 
		CDFs of $M^{\bm{X}}=\max \left( X_{1},\dots ,X_{d}\right) $ and $m^{\bm{X}
		}=\min \left( X_{1},\dots ,X_{d}\right) $ are  
		\begin{equation*}
			F_{M^{\bm{X}}}\left( x\right) =\sum_{y=0}^{x}\frac{\theta _{0}^{y}e^{-\theta
					_{0}}\prod_{j=1}^{d}\Gamma \left( x-y+1,\theta _{j}\right) }{\Gamma \left(
				y+1\right) \Gamma ^{d}\left( x-y+1,\right) },
			\quad
			F_{m^{\bm{X}}}\left( x\right) =\sum_{y=0}^{x}\frac{\theta _{0}^{y}e^{-\theta
					_{0}}\left[ \Gamma ^{d}\left( x-y+1\right) -\prod_{j=1}^{d}\gamma \left(
				x-y+1,\theta _{j}\right) \right] }{\Gamma \left( y+1\right) \Gamma
				^{d}\left( x-y+1\right) },
		\end{equation*}
		where $x\in \mathbb{N}$. 
	\end{theorem}
	
	\begin{proof}
		For the CDF of $M^{\bm{X}}$, conditioning on $Y_0=y_0$, we have,  
		\begin{align*}
			F_{M^{\bm{X}}}\left(x\right) 
			&=\sum_{y=0}^{x}\left[\mathbb{P}\left(Y_0=y\right)\prod_{j=1}^{d}\mathbb{P}
			\left(Y_j \le x-y\right)\right]=\sum_{y=0}^{x}\frac{\theta_{0}^{y}e^{-\theta_0}\prod_{j=1}^{d}\Gamma
				\left(x-y+1,\theta_j\right)}{\Gamma\left(y+1\right)\Gamma^d\left(x-y+1%
				\right) },
		\end{align*}
		where the last equality is from Lemma \ref{poisson lemma}. Similarly, for 
		the CDF of $m^{\bm{X}}$, we have,  
		\begin{align*}
			F_{m^{\bm{X}}}\left(x\right) 
			=\sum_{y=0}^{x}\frac{\theta_{0}^{y}e^{-\theta_0}\left[\Gamma^d\left(x-y+1
				\right)-\prod_{j=1}^{d}\gamma\left(x-y+1,\theta_j\right)\right]}{
				\Gamma\left(y+1\right)\Gamma^d\left(x-y+1\right)}.
		\end{align*}
	\end{proof}
	
	\subsection{Comonotonic shock structure}
	
	Similar to the common shock structure, we can utilize the conditional technique to explicitly derive the CDFs of $M^{\bm{X}}$ and $m^{\bm{X}}$ again. However, replacing $\bm{X}$'s single common shock variable with a $d$-dimensional comonotonic shock vector adds additional complexity to the derivation of the distributions of $M^{\bm{X}}$ and $m^{\bm{X}}$, as detailed in the following theorem.
	
	\begin{theorem}
		\label{comonotonic th} Suppose $\bm{X} = \left(X_1,\dots,X_d\right)$ follows
		the multivariate comonotonic shock Poisson distribution defined in 
		Definition \ref{def-comonotonic shock}, then the CDFs of $M^{\bm{X}
		}=\max\left(X_1,\dots,X_d\right)$ and $m^{\bm{X}}=\min\left(X_1,\dots,X_d
		\right)$ are
		\begin{align}  \label{comonotonic M eq1}
			F_{M^{\bm{X}}}\left(x\right)=\sum_{z_1=0}^{x}\cdots
			\sum_{z_d=0}^{x}\prod_{j=1}^{d}g_{\left(1-\theta\right)\lambda_j}\left(x-z_j
			\right)\min\left\{G_{\theta \lambda_1}\left(z_1\right),\dots,G_{\theta
				\lambda_d}\left(z_d\right)\right\},
		\end{align}
		or equivalently,  
		\begin{align}  \label{comonotonic M eq2}
			F_{M^{\bm{X}}}\left(x\right)=\sum_{z_1=0}^{x}\cdots
			\sum_{z_d=0}^{x}\prod_{j=1}^{d}G_{\left(1-\theta\right)\lambda_j}\left(x-z_j
			\right)C_{\Lambda,\theta}\left(z_1,\dots,z_d\right),
		\end{align}
		and  
		\begin{align}  \label{comonotonic m eq1}
			F_{m^{\bm{X}}}\left(x\right)=\sum_{z_1=0}^{\infty}\cdots
			\sum_{z_d=0}^{\infty}\prod_{j=1}^{d}g_{\left(1-\theta\right)\lambda_j}
			\left(z_j\right)\max\left\{G_{\theta
				\lambda_1}\left(x-z_1\right),\dots,G_{\theta
				\lambda_d}\left(x-z_d\right)\right\},
		\end{align}
		or equivalently,  
		\begin{align}  \label{comonotonic m eq2}
			F_{m^{\bm{X}}}\left(x\right)=1-\sum_{z_1=0}^{\infty}\cdots
			\sum_{z_d=0}^{\infty}\prod_{j=1}^{d}\overline{G}_{\left(1-\theta\right)
				\lambda_j}\left(x-z_j\right)C_{\Lambda,\theta}\left(z_1,\dots,z_d\right),
		\end{align}
		where $x \in \mathbb{N}$, $g_{\lambda}\left(x\right)$ denotes the PMF of the
		Poisson distribution with parameter $\lambda$,  
		\begin{equation*}
			C_{\Lambda,\theta}\left(z_1,\dots,z_d\right)=\left[\min\left\{G_{\theta
				\lambda_1}\left(z_1\right),\dots,G_{\theta
				\lambda_d}\left(z_d\right)\right\}-\max\left\{G_{\theta
				\lambda_1}\left(z_1-1\right),\dots,G_{\theta
				\lambda_d}\left(z_d-1\right)\right\}\right]_{+},
		\end{equation*}
		and $[\cdot]_+ = \max\{0,\cdot\}$. 
	\end{theorem}
	
	\begin{proof}
		Firstly, utilizing the conditional technique on $\bm{Z}$ or $\bm{Y}$, the CDF of $M^{\bm{X}}$ can be written as the following two forms:
		\begin{equation}  \label{cdf comonotonic3.2eq1}
			\begin{split}
				F_{M^{\bm{X}}}\left(x\right)
				&=\sum_{z_1=0}^{x}\cdots\sum_{z_d=0}^{x}\mathbb{P}\left(Z_1=z_1,
				\dots,Z_d=z_d\right)\mathbb{P}\left(Y_1 \le x-z_1,\dots,Y_d \le x-z_d\right)
				\\
				&=\sum_{z_1=0}^{x}\cdots\sum_{z_d=0}^{x}\mathbb{P}\left(Y_1=x-z_1,
				\dots,Y_d=x-z_d\right)\mathbb{P}\left(Z_1 \le z_1,\dots,Z_d \le z_d\right).
			\end{split}%
		\end{equation}
		The key point of above formula is to calculate $\mathbb{P}
		\left(Z_1=z_1,\dots,Z_d=z_d\right)$ and $\mathbb{P}\left(Z_1 \le 
		z_1,\dots,Z_d \le z_d\right)$. To obtain (\ref{comonotonic M eq1}), we 
		calculate the first one,  
		\begin{align*}
			C_{\Lambda,\theta}\left(z_1,\dots,z_d\right)&=\mathbb{P}\left(Z_1=z_1,
			\dots,Z_d=z_d\right) \\
			&=\mathbb{P}\left(U \in \left(G_{\theta
				\lambda_1}\left(z_1-1\right),G_{\theta \lambda_1}\left(z_1\right)\right]
			,\dots,U \in \left(G_{\theta \lambda_d}\left(z_d-1\right),G_{\theta
				\lambda_d}\left(z_d\right)\right]\right) \\
			&=\mathbb{P}\left(\max\left\{G_{\theta
				\lambda_1}\left(z_1-1\right),\dots,G_{\theta
				\lambda_d}\left(z_d-1\right)\right\}<U \le \min\left\{G_{\theta
				\lambda_1}\left(z_1\right),\dots,G_{\theta
				\lambda_d}\left(z_d\right)\right\}\right) \\
			&=\left[\min\left\{G_{\theta \lambda_1}\left(z_1\right),\dots,G_{\theta
				\lambda_d}\left(z_d\right)\right\}-\max\left\{G_{\theta
				\lambda_1}\left(z_1-1\right),\dots,G_{\theta
				\lambda_d}\left(z_d-1\right)\right\}\right]_{+},
		\end{align*}
		where $U$ is a uniform random variable and $U\sim\mathcal{U}(0,1)$.
		
		To obtain (\ref{comonotonic M eq2}), we calculate the second one, 
		\begin{align*}
			\mathbb{P}\left(Z_1 \le z_1,\dots,Z_d \le z_d\right)=\min\left\{G_{\theta \lambda_1}\left(z_1\right),\dots,G_{\theta
				\lambda_d}\left(z_d\right)\right\}.
		\end{align*}
		
		Secondly, for the CDF of $m^{\bm{X}}$, we have,  
		\begin{align}
			F_{m^{\bm{X}}}\left(x\right)
			&=1-\sum_{z_1=0}^{\infty}\cdots\sum_{z_d=0}^{\infty}\mathbb{P}
			\left(Z_1=z_,\dots,Z_d=z_d\right)\mathbb{P}\left(Y_1 \ge x-z_1+1,\dots,Y_d
			\ge x-z_d+1\right) \notag\\
			&=1-\sum_{z_1=0}^{\infty}\cdots \sum_{z_d=0}^{\infty}\mathbb{P}
			\left(Y_1=z_1,\dots,Y_d=z_d\right)\mathbb{P}\left(Z_1 \ge x-z_1+1,\dots,z_d
			\ge x-z_d+1\right).\label{cdf comonotonic3.2eq2}
		\end{align}
		
		Meanwhile, we have
		\begin{align*}
			\mathbb{P}\left(Z_1 \ge z_1,\dots,Z_d \ge z_d\right)=1-\max\left\{G_{\theta \lambda_1}\left(z_1-1\right),\dots,G_{\theta
				\lambda_d}\left(z_d-1\right)\right\},
		\end{align*}
		and  
		\begin{equation*}
			1=\sum_{z_1=0}^{\infty}\cdots
			\sum_{z_d=0}^{\infty}\prod_{j=1}^{d}g_{\left(1-\theta\right)\lambda_j}
			\left(z_j\right)=\sum_{z_1=0}^{\infty}g_{\left(1-\theta\right)\lambda_1}
			\left(z_1\right)\cdots
			\sum_{z_d=0}^{\infty}g_{\left(1-\theta\right)\lambda_d}\left(z_d\right).
		\end{equation*}
		Therefore, we can derive two forms of expression for $F_{m^{\bm{X}
		}}\left(x\right)$: 
		For (\ref{comonotonic m eq1}),  
		\begin{align*}
			F_{m^{\bm{X}}}\left(x\right)&=1-\sum_{z_1=0}^{\infty}\cdots
			\sum_{z_d=0}^{\infty}\mathbb{P}\left(Y_1=z_1,\dots,Y_d=z_d\right)\mathbb{P}
			\left(Z_1 \ge x-z_1+1,\dots,z_d \ge x-z_d+1\right) &  \\
			&=\sum_{z_1=0}^{\infty}\cdots
			\sum_{z_d=0}^{\infty}\prod_{j=1}^{d}g_{\left(1-\theta\right)\lambda_j}
			\left(z_j\right)\max\left\{G_{\theta
				\lambda_1}\left(x-z_1\right),\dots,G_{\theta
				\lambda_d}\left(x-z_d\right)\right\}. & 
		\end{align*}
		For (\ref{comonotonic m eq2}),  
		\begin{align*}
			F_{m^{\bm{X}}}\left(x\right)&=1-\sum_{z_1=0}^{\infty}\cdots\sum_{z_d=0}^{
				\infty}\mathbb{P}\left(Z_1=z_,\dots,Z_d=z_d\right)\mathbb{P}\left(Y_1 \ge
			x-z_1+1,\dots,Y_d \ge x-z_d+1\right) \\
			&=1-\sum_{z_1=0}^{\infty}\cdots \sum_{z_d=0}^{\infty}\prod_{j=1}^{d} 
			\overline{G}_{\left(1-\theta\right)\lambda_j}\left(x-z_j\right)C_{\Lambda,
				\theta}\left(z_1,\dots,z_d\right).
		\end{align*}
	\end{proof}

	Notably, in (\ref{cdf comonotonic3.2eq1}) and (\ref{cdf comonotonic3.2eq2}), we equivalently express the events $\{Y_1+Z_1\le x_1,\dots,Y_d+Z_d\le x_d\}$ and $\{Y_1+Z_1\ge x_1+1,\dots,Y_d+Z_d\ge x_d+1\}$ in two different forms, respectively. Consequently, we analytically derive two distinct CDFs for $M^{ \bm{X}}$, specifically (\ref{comonotonic M eq1}) and (\ref{comonotonic M eq2} ). Similarly, for $m^{\bm{X}}$, we derive (\ref{comonotonic m eq1}) and (\ref{comonotonic m eq2}). Although these expressions are equivalent, one may be more convenient for specific asymptotic analyses. For example, we use (\ref{comonotonic M eq2}) in Proposition \ref{as an comonotonic prop2}, while  Propositions \ref{as an comonotonic prop1} and \ref{as an comonotonic prop3}  utilize (\ref{comonotonic M eq1}).
	
	\subsection{Thinning-dependence structure}
	
	As illustrated in Definition \ref{def-thinning-dependence}, for a multivariate thinning-dependence Poisson distributed random vector $\bm{X}$, the components of marginal random variable $X_j$ are controlled by background random variable $Y_k$ (i.e., $X_j^k \vert Y_k = y_k \sim\text{Binomial}(p_j^k, y_k)$). This mechanism links all the marginal random variables through $l$ background random variables $Y_1, \dots, Y_l$ indirectly. Therefore, it is necessary to first investigate the specific mechanism by which the background random variables influence the marginal random variables, i.e., the conditional marginal random variable $X_j \vert Y_1 = y_1, \cdots, Y_l = y_l$. These results are presented in the following lemmas. The proof of Lemma \ref{lemma-3.2} is straightforward and thus omitted here.
	
	\begin{lemma}
		\label{lemma-3.2} Let $\bm{X} = \left(X_1,\dots,X_d\right)$ follows the 
		multivariate thinning-dependence Poisson distribution, then $X_j,j=1\dots,d$
		are mutually independent given $Y_1=y_1, \dots, Y_l=y_l$. 
	\end{lemma}
	
	$X_{j}|Y_{1}=y_{1},\cdots ,Y_{l}=y_{l}$ may involve the sum of independent, non-identically distributed binomial random variables, complicating the derivation of the CDFs for $M^{\bm{X}}$ and $m^{\bm{X}}$ under the thinning-dependence structure. Hence, we first provide the explicit PMF of these conditional marginal random variables in the following lemma.
	\begin{lemma}
		\label{lemma-3.3} Let $\bm{X} = \left(X_1,\dots,X_d\right)$ follows the 
		multivariate thinning-dependence Poisson distribution defined in Definition  %
		\ref{def-thinning-dependence}. Then, for any $j=1,\dots,d$, when $s\le 
		y_1+\dots+y_l$, we have,  
		\begin{equation}  \label{eq lemma3.3}
			\begin{aligned} & P\left(X_j=s \vert
				Y_1=y_1,\cdots,Y_l=y_l\right)
				\\
				&=\sum_{x_j^1=\max \left(0, s-\sum_{i=2}^l
					y_i\right)}^{\min \left(s, y_1\right)}\left\{P\left(X_j^1=x_j^1|Y_1=y_1\right)
				\sum_{x_j^2=\max \left(0, s-x_j^1-\sum_{i=3}^l y_i\right)}^{\min
					\left(s-x_j^1, y_2\right)} P\left(X_j^2=x_j^2|Y_2=y_2\right)\right. \\ &
				\ldots\left[\sum_{x_j^k=\max \left(0, s-\sum_{i=1}^{k-1}
					x_j^i-\sum_{i=k+1}^l y_i\right)}^{\min \left(s-\sum_{i=1}^{k-1} x_j^i,
					y_k\right)} P\left(X_j^k=x_j^k|Y_k=y_k\right)\right. \\ &
				\left.
				\left.\ldots\left(\sum_{x_j^{l-1}=\max \left(0, s-\sum_{i=1}^{l-2}
					x_j^i-y_l\right)}^{\min \left(s-\sum_{i=1}^{l-2} x_j^i, y_{l-1}\right)}
				P\left(X_j^{l-1}=x_j^{l-1}|Y_{l-1}=y_{l-1}\right) P\left(X_j^l=s-\sum_{i=1}^{l-1}
				x_j^i|Y_l=y_l\right)\right)\right]
				\right\}
				. \end{aligned}
		\end{equation}
	\end{lemma}
	
	\begin{proof}
		From Definition \ref{def-thinning-dependence}, we can deduce that, given $Y_{1}=y_{1},\dots ,Y_{l}=y_{l}$, the random variables $X_{j}^{k}\sim\text{Binomial}(y_{k},p_{j}^{k})$, $k=1,\dots ,l$, are independent. Thus, for any $j=1,\dots ,d$, $X_{j}|Y_{1}=y_{1},\dots ,Y_{l}=y_{l}$ is the sum of $d$ independent non-identically distributed binomial random variables, and its PMF is given by \cite{eisinga2013saddlepoint}. 
	\end{proof}
	
	Using the conditional marginal distribution and employing the conditional 
	technique once more, we can derive the following theorem.
	
	\begin{theorem}
		Suppose $\bm{X} = \left(X_1,\dots,X_d\right)$ follows the multivariate 
		thinning-dependence Poisson distribution defined in Definition \ref%
		{def-thinning-dependence}, then the CDFs of $M^{\bm{X}}=\max\left(X_1,
		\dots,X_d\right)$ and $m^{\bm{X}}=\min\left(X_1,\dots,X_d\right)$ are  
		\begin{equation}  \label{thinning cdf M}
			F_{M^{\bm{X}}}\left(x\right)=\sum_{y_1+\cdots+y_l \le x}^{}\prod_{k=1}^{l} 
			\frac{\theta_k^{y_k}e^{-\theta_k}}{\Gamma\left(y_k+1\right)}
			+\sum_{y_1+\cdots+y_l \ge x+1}^{}\left\{\left[\prod_{j=1}^{d}
			\sum_{s=0}^{x}h_j\left(s,y_1,\dots,y_l\right)\right]\left[\prod_{k=1}^{l} 
			\frac{\theta_k^{y_k}e^{-\theta_k}}{\Gamma\left(y_k+1\right)}\right]\right\}
		\end{equation}
		and  
		\begin{equation}  \label{thinning cdf m}
			F_{m^{\bm{X}}}\left(x\right)=\sum_{y_1+\cdots+y_l \le x}^{}\prod_{k=1}^{l} 
			\frac{\theta_k^{y_k}e^{-\theta_k}}{\Gamma\left(y_k+1\right)}
			+\sum_{y_1+\cdots+y_l \ge x+1}^{}\left\{\left[1-\prod_{j=1}^{d}\left(1-
			\sum_{s=0}^{x}h_j\left(s,y_1,\dots,y_l\right)\right)\right]\left[
			\prod_{k=1}^{l}\frac{\theta_k^{y_k}e^{-\theta_k}}{\Gamma\left(y_k+1\right)} %
			\right]\right\},
		\end{equation}
		where $h_j\left(s,y_1,\dots,y_l\right)=P\left(X_j=s \vert 
		Y_1=y_1,\cdots,Y_l=y_l\right)$, as explicitly given in Lemma \ref{lemma-3.3}%
		. 
	\end{theorem}
	
	\label{thinning th}
	
	\begin{proof}
		For $F_{M^{\bm{X}}}(x)$, we have,
		\begin{align*}
			F_{M^{\bm{X}}}\left(x\right) &=\mathbb{P}\left(X_1 \le x,\cdots,X_d \le
			x\right) \\
			&=\sum_{y_k=0,k=1,\dots,l}^{\infty}\mathbb{P}\left(Y_1=y_1,\cdots,Y_l=y_l
			\right)\mathbb{P}\left(X_1 \le x, \dots,X_d \le x \vert
			Y_1=y_1,\dots,Y_l=y_l\right) \\
			&=\sum_{y_1+\cdots+y_l \le x}^{}\prod_{k=1}^{l}\mathbb{P}\left(Y_k=y_k
			\right)+\sum_{y_1+\cdots+y_l \ge x+1}^{}\prod_{j=1}^{d}\mathbb{P}
			\left(X_j\le x \vert Y_1=y_1,\cdots,Y_l=y_l\right)\prod_{k=1}^{l}\mathbb{P}
			\left(Y_k=y_k\right) \\
			&=\sum_{y_1+\cdots+y_l \le x}^{}\prod_{k=1}^{l}\frac{\theta_k^{y_k}e^{-
					\theta_k}}{\Gamma\left(y_k+1\right)}+\sum_{y_1+\cdots+y_l \ge x+1}^{}\left\{ %
			\left[\prod_{j=1}^{d}\sum_{s=0}^{x}h_j\left(s,y_1,\dots,y_m\right)\right] %
			\left[\prod_{k=1}^{l}\frac{\theta_k^{y_k}e^{-\theta_k}}{\Gamma\left(y_k+1
				\right)}\right]\right\},
		\end{align*}
		For $F_{m^{\bm{X}}}(x)$, 
		we have,  
		\begin{align*}
			&F_{m^{\bm{X}}}\left(x\right) = 1-\mathbb{P}\left(X_1>x,\dots,X_d>x\right) & 
			\\
			&=\sum_{y_1+\cdots+y_l \ge x+1}^{}\left\{\left[1-\prod_{j=1}^{d}\left(1-
			\sum_{s=0}^{x}h_j\left(s,y_1,\dots,y_l\right)\right)\right]\left[
			\prod_{k=1}^{l}\frac{\theta_k^{y_k}e^{-\theta_k}}{\Gamma\left(y_k+1\right)} %
			\right]\right\} +\sum_{y_1+\cdots+y_l \le x}^{}\prod_{k=1}^{l}\frac{\theta_k^{y_k}e^{-
					\theta_k}}{\Gamma\left(y_k+1\right)}， & 
		\end{align*}
		which completes the proof.
		
	\end{proof}
	
	When the number of background random variables $l\geq 2$, the conditional marginal distribution (\ref{eq lemma3.3}) involves the sum of independent, non-identically distributed binomial random variables, making the explicit expressions (\ref{thinning cdf M}) and (\ref{thinning cdf m}) more complex. By contrast, in the case of $l=1$, an elegant formula employing the regularized incomplete beta function is presented in the following corollary.
	
	\begin{corollary}
		Suppose $\bm{X}=\left( X_{1},\dots ,X_{d}\right) $ follows the multivariate 
		thinning-dependence Poisson distribution defined in Definition \ref%
		{def-thinning-dependence} with $l=1$, then the CDFs of $M^{\bm{X}}=\max 
		\left( X_{1},\dots ,X_{d}\right) $ and $m^{\bm{X}}=\min \left( X_{1},\dots 
		,X_{d}\right) $ are  
		\begin{equation*}
			F_{M^{\bm{X}}}\left( x\right) =\frac{\Gamma \left( x+1,\theta _{1}\right) }{
				\Gamma \left( x+1\right) }+\sum_{y=x+1}^{\infty }\frac{\theta
				_{1}^{y}e^{-\theta _{1}}}{\Gamma \left( y+1\right) }
			\prod_{j=1}^{d}I_{1-p_{j}}\left( y-x,x+1\right) ,
		\end{equation*}
		and  
		\begin{equation*}
			F_{m^{\bm{X}}}\left( x\right) =\frac{\Gamma \left( x+1,\theta _{1}\right) }{
				\Gamma \left( x+1\right) }+\sum_{y=x+1}^{\infty }\frac{\theta
				_{1}^{y}e^{-\theta _{1}}}{\Gamma \left( y+1\right) }\left[
			1-\prod_{j=1}^{d}I_{p_{j}}\left( x+1,y-x\right) \right] ,
		\end{equation*}
		where $I_{x}\left( a,b\right) =\frac{\int_{0}^{x}t^{a-1}(1-t)^{b-1}dt}{
			\int_{0}^{1}t^{a-1}(1-t)^{b-1}dt}$ is regularized incomplete beta function. 
	\end{corollary}
	
	\section{Asymptotic analysis}\label{sec4}
	
	In this section, we conduct asymptotic analyses for the CDFs of the maximum and minimum of the aforementioned three multivariate Poisson distributions. Specifically, we work out two types of asymptotic results; one investigates the derived CDFs' asymptotic behaviors when the dimension $d\rightarrow \infty $, and the other examines scenarios where certain parameters (excluding $d$) dominate the dependence structure. We indeed obtain some interesting results that reveal the essential difference between the three dependence structures. Compared to the existing literature, our results offer an alternative approach to analyzing dependent structures from the perspective of the distribution of maximum (minimum) values and their asymptotic properties. For instance, the gap in the asymptotic results for dimension $d$ between the common shock and comonotonic shock structures reveals a crucial difference in extending the common shock variable to a comonotonic shock vector, and we provide a counterexample to highlight this point. These findings supplement the work of \cite{schulz2021multivariate} and further showcase the theoretical extension from common shock to comonotonic shocks.
	
	We first define the asymptotic equivalence as follows.
	
	
	\begin{definition}
		Given functions $f(x)$ and $g(x)$, we define a binary relation  
		\begin{equation*}
			f(x;\kappa)\sim g(x;\kappa) ~~(\text{as}~ \kappa\rightarrow \kappa_0)~\text{ if and only if }\lim\limits_{\kappa\rightarrow \kappa_0}\frac{f(x;\kappa)}{g(x;\kappa)} = 1.
		\end{equation*}
	\end{definition}
	
	\begin{remark}
		For the sake of clarity, we further introduce following notations. For  arbitrary random vector $\bm{\xi} = \left(\xi_1,\dots,\xi_d\right)$, denote  
		\begin{align*}
			M_{-j}^{\bm{\xi}}=\max\left(\xi_1,\dots,\xi_{j-1},\xi_{j+1},\dots,\xi_d
			\right),\quad m_{-j}^{\bm{\xi}}=\min\left(\xi_1,\dots,\xi_{j-1},\xi_{j+1},
			\dots,\xi_d\right).
		\end{align*}
	\end{remark}
	
	
	\subsection{Common shock model}
	
	In this subsection, we will discuss the asymptotic analysis for the Poisson 
	rates of $Y_j$, $j=1,\dots,d$, and $Y_0$ for the multivariate common shock Poisson distribution, as presented in Proposition \ref{as an common prop1} and Proposition \ref{as an common prop2}, respectively. Moreover, we consider the asymptotic result for the dimension of $\bm{X}$ in Proposition \ref{as an common prop3}.
	
	\begin{proposition}
		\label{as an common prop1} Let $\bm{X}=(X_1,\dots,X_d)$ follows the 
		multivariate common shock Poisson distribution defined in Definition \ref%
		{def-common-shock}, then  
		\begin{align*}
			F_{M^{\bm{X}}}\left(x\right) \sim F_{X_i}\left(x\right)F_{M_{-i}^{\bm{Y}
			}}\left(x\right), \quad F_{m^{\bm{X}}}(x) \sim F_{m_{-i}^{\bm{X}}}(x),~\text{as}~\theta_{i}\rightarrow\infty.
		\end{align*}
	\end{proposition}
	
	\begin{proof}
		Without loss of generality, we assume $i=d$, and the results for other cases can be similarly derived. Firstly, we consider the maximum distribution $F_{M^{\bm{X}}}(x)$. From the Definition \ref{def-common-shock}, the marginal distribution of $\bm{X}$ is  
		\begin{equation*}
			F_{X_i}\left(x\right)=\sum_{y=0}^{x}\frac{e^{-\theta_0}\theta_{0}^{y}\Gamma
				\left(x-y+1,\theta_i\right)}{\Gamma\left(y+1\right)\Gamma\left(x-y+1\right)}.
		\end{equation*}
		Therefore,  
		\begin{equation*}
			\lim\limits_{\theta_d \to +\infty}\frac{F_{M^{\bm{X}}}\left(x\right)}{
				F_{X_d}\left(x\right)}=\lim\limits_{\theta_d \to +\infty}\frac{%
				\sum_{y=0}^{x} \frac{\theta_{0}^{y}e^{-\theta_0}\prod_{j=1}^{d}\Gamma%
					\left(x-y+1,\theta_j \right)}{\Gamma\left(y+1\right)\Gamma^d\left(x-y+1%
					\right)}}{\sum_{y=0}^{x} \frac{e^{-\theta_0}\theta_{0}^{y}\Gamma\left(x-y+1,%
					\theta_d\right)}{ \Gamma\left(y+1\right)\Gamma\left(x-y+1\right)}}.
		\end{equation*}
		It is noteworthy that, for any given $z=0,\dots,x$, we have,  
		\begin{equation}  \label{as an common eq1}
			\lim\limits_{\theta_d \to +\infty}\frac{\frac{\theta_{0}^{z}e^{-\theta_0}
					\prod_{j=1}^{d}\Gamma\left(x-z+1,\theta_j\right)}{\Gamma\left(z+1\right)
					\Gamma^d\left(x-z+1\right)}}{\sum_{y=0}^{x}\frac{e^{-\theta_0}\theta_{0}^{y}
					\Gamma\left(x-y+1,\theta_d\right)}{\Gamma\left(y+1\right)\Gamma\left(x-y+1
					\right)}}=\frac{\frac{\theta_{0}^{z}e^{-\theta_0}\Gamma(x-y+1)
					\prod_{j=1}^{d-1}\Gamma\left(x-z+1,\theta_j\right)}{\Gamma\left(z+1\right)
					\Gamma^{d}\left(x-z+1\right)}}{\sum_{y=0}^{x}\frac{\theta_{0}^{y}e^{-
						\theta_0}}{\Gamma\left(y+1\right)}\lim\limits_{\theta_d \to +\infty}\frac{
					\Gamma\left(x-y+1,\theta_d\right)}{\Gamma\left(x-z+1,\theta_d\right)}},
		\end{equation}
		and the limit part of the right side of (\ref{as an common eq1}) is  
		\begin{align*}
			\lim\limits_{\theta_d \to +\infty}\frac{\Gamma\left(x-y+1,\theta_d\right)}{
				\Gamma\left(x-z+1,\theta_d\right)}=\lim\limits_{\theta_d \to \infty}\frac{
				\int_{\theta_d}^{+\infty}t^{x-y}e^{-t}dt}{\int_{\theta_d}^{+
					\infty}t^{x-z}e^{-t}dt}= 
			\begin{cases}
				\begin{aligned} &0, \quad &&\text{if}~z<y, \\ &1, \quad &&\text{if} ~z=y, \\
					&+\infty, \quad &&\text{if} ~z>y. \end{aligned}%
			\end{cases}
			&
		\end{align*}
		Hence,
		
		\begin{align*}
			\lim\limits_{\theta_d \to +\infty}\frac{F_{M^{\bm{X}}}\left(x\right)}{
				F_{X_d}\left(x\right)} =\prod_{j=1}^{d-1}\frac{\Gamma\left(x+1,\theta_j
				\right)}{\Gamma\left(x+1\right)} = F_{M_{-d}^{\bm{Y}}}(x).
		\end{align*}
		
		Secondly, for $F_{m^{\bm{X}}}(x)$, we can similarly deduce that
		\begin{align*}
			\lim\limits_{\theta_d \to \infty}\frac{F_{m^{\bm{X}}}\left(x\right)}{
				F_{m_{-d}^{\bm{X}}}\left(x\right)}&=\lim\limits_{\theta_d \to \infty}\frac{
				\sum_{y=0}^{x}\frac{\theta_{0}^{y}e^{-\theta_0}\left[\Gamma^d\left(x-y+1
					\right)-\prod_{j=1}^{d}\gamma\left(x-y+1,\theta_j\right)\right]}{
					\Gamma\left(y+1\right)\Gamma^d\left(x-y+1\right)}}{\sum_{y=0}^{x}\frac{
					\theta_{0}^{y}e^{-\theta_0}\left[\Gamma^{d-1}\left(x-y+1\right)-
					\prod_{j=1}^{d-1}\gamma\left(x-y+1,\theta_j\right)\right]}{
					\Gamma\left(y+1\right)\Gamma^{d-1}\left(x-y+1\right)}} =1.
		\end{align*}
	\end{proof}
	
	The following proposition analyzes the asymptotic equivalence of $F_{M^{\bm{X}}}(x)$ and $F_{m^{\bm{X}}}(x)$ as the rate of the common shock variable $Y_0$ approaches infinity.
	
	\begin{proposition}
		\label{as an common prop2} Let $\bm{X}=(X_1,\dots,X_d)$ follows the 
		multivariate common shock Poisson distribution defined in Definition \ref%
		{def-common-shock}, then 
		\begin{equation*}
			F_{M^{\bm{X}}}\left(x\right) \sim
			e^{-\sum_{j=1}^{d}\theta_j}F_{Y_0}(x),\quad F_{m^{\bm{X}}}\left(x\right)
			\sim \left[1-\prod_{j=1}^{d}\left(1-e^{-\theta_j}\right)\right]F_{Y_0}(x),~\text{as}~\theta_{0}\rightarrow\infty.
		\end{equation*}
	\end{proposition}
	
	\begin{proof}
		For $F_{M^{\bm{X}}}(x)$, we have
		
		\begin{equation}  \label{as an common eq2}
			\begin{split}
				\lim\limits_{\theta_0 \to +\infty}\frac{F_{M^{\bm{X}}}\left(x\right)}{
					e^{-\sum_{j=1}^{d}\theta_j}F_{Y_0}(x)}&=\lim\limits_{\theta_0 \to
					+\infty}\sum_{y=0}^{x}\frac{\theta_{0}^{y}e^{-\theta_0}\prod_{j=1}^{d}\Gamma
					\left(x-y+1,\theta_j\right)}{\Gamma\left(y+1\right)\Gamma^d\left(x-y+1%
					\right) }\frac{\Gamma\left(x+1\right)}{e^{-\sum_{j=1}^{d}\theta_j}\Gamma%
					\left(x+1, \theta_0\right)} \\
				&=\sum_{y=0}^{x}\lim\limits_{\theta_0 \to +\infty}\frac{\theta_{0}^{y}e^{-
						\theta_0}}{\Gamma\left(x+1,\theta_0\right)}e^{\sum_{j=1}^{d}\theta_j}\frac{
					\Gamma\left(x+1\right)\prod_{j=1}^{d}\Gamma\left(x-y+1,\theta_j\right)}{
					\Gamma\left(y+1\right)\Gamma^d\left(x-y+1\right)}.
			\end{split}%
		\end{equation}
		The limit part of the right side of (\ref{as an common eq2}) is  
		\begin{equation}  \label{as an common eq3}
			\begin{split}
				\lim\limits_{\theta_0 \to +\infty}\frac{\theta_{0}^{y}e^{-\theta_0}}{
					\Gamma\left(x+1,\theta_0\right)} = \lim\limits_{\theta_0 \to +\infty}\frac{
					\theta_{0}^{y}e^{-\theta_0}}{\int_{\theta_0}^{+\infty}t^{x}e^{-t}dt}= 
				\begin{cases}
					0, & \text{if } y < x, \\ 
					1, & \text{if } y=x.%
				\end{cases}%
			\end{split}%
		\end{equation}
		Therefore,  
		\begin{align*}
			\lim\limits_{\theta_0 \to +\infty}\frac{F_{M}\left(x\right)}{
				e^{-\sum_{j=1}^{d}\theta_j}F_{Y_{0}}(x)}=e^{\sum_{j=1}^{d}\theta_j}\frac{
				\Gamma\left(x+1\right)\prod_{j=1}^{d}\Gamma\left(1,\theta_j\right)}{
				\Gamma\left(x+1\right)\Gamma^d\left(1\right)}=1.
		\end{align*}
		
		The proof of $F_{m^{\bm{X}}}(x)$ can be obtained in a similar way. 

	\end{proof}
	
	Before presenting the asymptotic analysis results concerning the dimension of $\bm{X}$, we first provide the following lemma, which can be directly derived from the properties of the Gamma function.
	
	\begin{lemma}
		\label{as an common lemma1} Sequence $a_n=\frac{\Gamma(n)}{\Gamma(n,k)}$ is 
		strictly monotonically decreasing.
	\end{lemma}
	
	\begin{proposition}
		\label{as an common prop3} Let $\bm{X}=(X_1,\dots,X_d)$ follows the 
		multivariate common shock Poisson distribution defined in Definition \ref%
		{def-common-shock}, then  
		\begin{equation*}
			F_{M^{\bm{X}}}\left(x\right) \sim F_{Y_0}(0)F_{M^{\bm{Y}}}\left(x\right),
			\quad F_{m^{\bm{X}}}\left(x\right) \sim F_{Y_0}(x),~\text{as}~d\rightarrow \infty.
		\end{equation*}
	\end{proposition}
	
	\begin{proof}
		On one hand, for $F_{M^{\bm{X}}}(x)$, our objective is to prove  
		\begin{equation}
			\lim\limits_{d \to \infty}\frac{F_{M^{\bm{X}}}\left(x\right)}{
				F_{Y_0}(0)F_{M^{\bm{Y}}}\left(x\right)} = \lim\limits_{d \to \infty}\frac{
				F_{M^{\bm{X}}}\left(x\right)}{e^{-\theta_0}F_{M^{\bm{Y}}}\left(x\right)}=1.
			\label{e15}
		\end{equation}
		Notably, $Y_j,j=1,\dots,d$ are mutually independent, thus,
		
		\begin{align*}
			F_{M^{\bm{Y}}}\left(x\right)&=\mathbb{P}\left(M^{\bm{Y}} \le x\right)
			=\prod_{j=1}^{d}F_{Y_{j}}\left( x\right) =\frac{\prod_{j=1}^{d}\Gamma
				\left(x+1,\theta_j\right)}{\Gamma^d\left(x+1\right)}.
		\end{align*}
		Therefore, we have,
		\begin{align*}
			\frac{F_{M^{\bm{X}}}\left(x\right)}{e^{-\theta_0}F_{M^{\bm{Y}%
				}}\left(x\right) }=1+\sum_{y=1}^{x}\frac{\theta_{0}^{y}}{\Gamma\left(y+1\right)}
			\prod_{j=1}^{d}\frac{\Gamma\left(x+1\right)\Gamma\left(x-y+1,\theta_j\right) 
			}{\Gamma\left(x-y+1\right)\Gamma\left(x+1,\theta_j\right)}. 
		\end{align*}
		Moreover, from Lemma \ref{as an common lemma1}, for $y=1, \dots,x$, we have, $\frac{\Gamma\left(x+1\right)\Gamma\left(x-y+1,\theta_j\right)}{
			\Gamma\left(x-y+1\right)\Gamma\left(x+1,\theta_j\right)}<1,$ which implies that $\lim\limits_{d \to \infty}\frac{F_{M^{\bm{X}}}\left(x\right)}{
			e^{-\theta_0}F_{M^{\bm{Y}}}\left(x\right)}=1.$ 
		
		The proof of $F_{m^{\bm{X}}}(x)$ follows a similar approach to that of $F_{M^{\bm{X}}}(x)$.

	\end{proof}
	
	\subsection{Comonotonic shock model}
	
	In this subsection, we discuss the asymptotic analysis of the derived CDFs for the multivariate comonotonic shock Poisson distribution. We focus on the parameters $\lambda_j$, $j=1,\dots,d$, as presented in Proposition \ref{as an comonotonic prop1}, and the parameter $\theta$ in Propositions \ref{as an comonotonic prop2} and \ref{as an comonotonic prop3}. Furthermore, we also consider the asymptotic result of the dimension of $\bm{X}$ in Proposition \ref{as an comonotonic prop4}.
	
	\begin{proposition}
		\label{as an comonotonic prop1} Let $\bm{X}=(X_1,\dots,X_d)$ follows the 
		multivariate comonotonic shock Poisson distribution defined in Definition  %
		\ref{def-comonotonic shock}, then  
		\begin{equation*}
			F_{M^{\bm{X}}}\left(x\right) \sim F_{X_i}\left(x\right)F_{M_{-i}^{\bm{Y}
			}}\left(x\right),\quad F_{m^{\bm{X}}}(x)\sim F_{m_{-i}^{\bm{X}}}(x),~\text{as}~\lambda_i\rightarrow\infty.
		\end{equation*}
	\end{proposition}
	
	\begin{proof}
		We first consider the marginal distribution of $\bm{X}$, it can be written 
		as:  
		\begin{align*}
			&F_{X_i}\left(x\right)=\mathbb{P}\left(X_i \le x\right)
			=\sum_{z_i=0}^{x}g_{\left(1-\theta\right)\lambda_i}\left(x-z_i\right)G_{
				\theta \lambda_i}\left(z_i\right). & 
		\end{align*}
		Notably, for any given $z_1,\dots,z_d$ and $z_i=0,1,\dots,x$,  
		\begin{equation*}
			\min\left\{G_{\theta \lambda_1}\left(z_1\right),\dots,G_{\theta
				\lambda_d}\left(z_d\right)\right\} \sim G_{\theta\lambda_{i}}(z_i), \quad ~ 
			\text{as} ~\lambda_i\rightarrow\infty,
		\end{equation*}
		hence, for the distribution of the maximum, $F_{M^{\bm{X}}}(x)$, we have:  
		\begin{align*}
			\lim\limits_{\lambda_i \to \infty}&\frac{F_{M^{\bm{X}}}\left(x\right)}{
				F_{X_i}\left(x\right)}=\lim\limits_{\lambda_i \to \infty}\frac{
				\sum_{z_1=0}^{x}\cdots \sum_{z_d=0}^{x} \min\left\{G_{\theta
					\lambda_1}\left(z_1\right),\dots,G_{\theta
					\lambda_d}\left(z_d\right)\right\}\prod_{j=1}^{d}g_{\left(1-\theta\right)
					\lambda_j}\left(x-z_j\right) }{\sum_{z_i=0}^{x}g_{\left(1-\theta\right)
					\lambda_i}\left(x-z_i\right)G_{\theta \lambda_i}\left(z_i\right)} &  \\
			&=\lim\limits_{\lambda_i \to \infty}\frac{\sum_{z_1=0}^{x}g_{\left(1-\theta
					\right)\lambda_1}\left(x-z_1\right)\dots
				\sum_{z_i=0}^{x}g_{\left(1-\theta\right)\lambda_i}\left(x-z_i\right)G_{
					\theta \lambda_i}\left(z_i\right) \dots
				\sum_{z_d=0}^{x}g_{\left(1-\theta\right)\lambda_d}\left(x-z_d\right)}{
				\sum_{z_i=0}^{x}g_{\left(1-\theta\right)\lambda_i}\left(x-z_i\right)G_{
					\theta \lambda_i}\left(z_i\right)} &  \\
		&=\prod_{j=1,j \neq i}^{d}G_{\left(1-\theta\right)\lambda_j}\left(x\right).
		& 
	\end{align*}
	Similarly, the asymptotic analysis of $F_{m^{\bm{X}}}$ can be proved using the same method.

	\end{proof}
	
	
	\begin{proposition}
	\label{as an comonotonic prop2} Let $\bm{X}=(X_1,\dots,X_d)$ follows the 
	multivariate comonotonic shock Poisson distribution defined in Definition  %
	\ref{def-comonotonic shock}, then  
	\begin{equation*}
		F_{M^{\bm{X}}}(x)\sim F_{M^{\bm{Y}}}(x),\quad F_{m^{\bm{X}}}(x)\sim F_{m^{ %
				\bm{Y}}}(x),~\text{as}~\theta\rightarrow0.
	\end{equation*}
	\end{proposition}
	
	\begin{proof}
	Firstly, according to (i) in Definition \ref{def-comonotonic shock}, the 
	CDFs of $M^{\bm{Y}}$ and $m^{\bm{Y}}$ are given by:  
	\begin{equation}  \label{as an comonotonic eq1}
		F_{M^{\bm{Y}}}(x)=\prod_{j=1}^{d}G_{\left(1-\theta\right)\lambda_j}\left(x
		\right),\quad F_{m^{\bm{Y}}}(x)=1-\prod_{j=1}^{d}\overline{G}
		_{\left(1-\theta\right)\lambda_j}\left(x\right).
	\end{equation}
	Meanwhile, we have:  
	\begin{equation}  \label{as comonotonic 4.5eq}
		\lim\limits_{\theta \to
			0}C_{\Lambda,\theta}\left(z_1,\dots,z_d\right)=\left\{ 
		\begin{array}{ll}
			1, & {z_1=\cdots=z_d=0}, \\ 
			0, & \text{otherwise}.%
		\end{array}
		\right.
	\end{equation}
	Thus, combining (\ref{as comonotonic 4.5eq}) with (\ref{comonotonic M eq2}),
	(\ref{comonotonic m eq2}) and (\ref{as an comonotonic eq1}), we have  
	\begin{align*}
		\lim\limits_{\theta \to 0}\frac{F_{M^{\bm{X}}}\left(x\right)}{ F_{M^{\bm{Y}
			}}(x)}=\lim\limits_{\theta \to
			0}\sum_{z_1=0}^{x}\cdots\sum_{z_d=0}^{x}\prod_{j=1}^{d}\frac{
			G_{\left(1-\theta\right)\lambda_j}\left(x-z_j\right)}{G_{\left(1-\theta
				\right)\lambda_j}\left(x\right)}C_{\Lambda,\theta}\left(z_1,\dots,z_d\right)
		=\prod_{j=1}^{d}\frac{G_{\left(1-\theta\right)\lambda_j}\left(x\right)}{
			G_{\left(1-\theta\right)\lambda_j}\left(x\right)}=1,
	\end{align*}
	and  
	\begin{align*}
		\lim\limits_{\theta \to 0}\frac{F_{m^{\bm{X}}}\left(x\right)}{F_{m^{\bm{Y}
			}}(x)}&=\lim\limits_{\theta \to 0}\frac{1-\sum_{z_1=0}^{\infty}\cdots
			\sum_{z_d=0}^{\infty}\prod_{j=1}^{d}\overline{G}_{\left(1-\theta\right)
				\lambda_j}\left(x-z_j\right)C_{\Lambda,\theta}\left(z_1,\dots,z_d\right)}{
			1-\prod_{j=1}^{d}\overline{G}_{\left(1-\theta\right)\lambda_j}\left(x\right)}
		=1. & 
	\end{align*}
	\end{proof}
	
	\begin{proposition}
	\label{as an comonotonic prop3} Let $\bm{X}=(X_1,\dots,X_d)$ follows the 
	multivariate comonotonic shock Poisson distribution defined in Definition  %
	\ref{def-comonotonic shock}, then  
	\begin{equation*}
		F_{M^{\bm{X}}}(x)\sim F_{M^{\bm{Z}}}(x),\quad F_{m^{\bm{X}}}(x)\sim F_{m^{ %
				\bm{Z}}}(x),~\text{as}~\theta\rightarrow1.
	\end{equation*}
	\end{proposition}
	
	\begin{proof}
	Firstly, according to (ii) in Definition \ref{def-comonotonic shock}, the 
	CDFs of $M^{\bm{Z}}$ and $m^{\bm{Z}}$ are:  
	\begin{equation}  \label{as an comonotonic eq2}
		F_{M^{\bm{Z}}}\left(x\right)=\min\left\{G_{\theta
			\lambda_1}\left(x\right),\dots,G_{\theta
			\lambda_d}\left(x\right)\right\},\quad F_{m^{\bm{Z}}}\left(x\right)=\max
		\left\{G_{\theta \lambda_1}\left(x\right),\dots,G_{\theta
			\lambda_d}\left(x\right)\right\}.
	\end{equation}
	Meanwhile, we have  
	\begin{equation}  \label{as comonotonic 4.6eq}
		\lim\limits_{\theta \to
			1}\prod_{j=1}^{d}g_{\left(1-\theta\right)\lambda_j}\left(x-z_j\right)=\left
		\{ 
		\begin{array}{ll}
			1, & z_1=\cdots=z_d=x, \\ 
			0, & \text{otherwise}.%
		\end{array}
		\right.
	\end{equation}
	Thus, combining (\ref{as comonotonic 4.6eq}) with (\ref{comonotonic M eq1}) and (\ref{as an comonotonic eq2}), we have  
	\begin{align*}
		\lim\limits_{\theta \to 1}\frac{F_{M^{\bm{X}}}\left(x\right)}{F_{M^{\bm{Z}
			}}(x)}&=\lim\limits_{\theta \to 1}\frac{\sum_{z_1=0}^{x}\cdots
			\sum_{z_d=0}^{x}\prod_{j=1}^{d}g_{\left(1-\theta\right)\lambda_j}\left(x-z_j
			\right)\min\left\{G_{\theta \lambda_1}\left(z_1\right),\dots,G_{\theta
				\lambda_d}\left(z_d\right)\right\}}{\min\left\{G_{\theta
				\lambda_1}\left(x\right),\dots,G_{\theta \lambda_d}\left(x\right)\right\}} \\
		&=\frac{\min\left\{G_{\lambda_1}\left(x\right),\dots,G_{\lambda_d}\left(z_d
			\right)\right\}}{\min\left\{G_{\lambda_1}\left(x\right),\dots,G_{\lambda_d}
			\left(z_d\right)\right\}} =1.
	\end{align*}
	and the proof of $F_{m^{\bm{X}}}(x)$ follows a similar approach to that of $F_{M^{\bm{X}}}(x)$.
		\end{proof}
		
		Next, we give the asymptotic analysis concerning the dimension $d$. We first
		introduce some notations. For a multivariate comonotonic shock Poisson 
		distribution $\bm{X}=(X_1,\dots,X_d)$, we denote that $$G_{x,d}:=\{G_{\theta
	\lambda_{1}}(x),\dots,G_{\theta\lambda_{d}}(x)\},$$ for $x=0,1,\dots$ and $
	\Lambda=\lim\limits_{d\to\infty}\Lambda_d$. 
	Thus, we denote the bonded infinite set of real numbers $G^{(1)}:=\lim\limits_{d\rightarrow
	\infty}G_{0,d}=\{G_{\theta\lambda_{1}}(0),G_{\theta\lambda_{2}}(0),\dots\}$,
	its infimum is denoted as $m^{(1)}=\inf G^{(1)}$. For $k=2,3,\dots$, we denote that, $G^{(k)} := G^{(k-1)} \setminus \{m^{(k-1)}\}$, $m^{(k)} := \inf G^{(k)}.$
	Similarly, we denote the bonded infinite set of real numbers $\tilde{G}
	^{(1)}_x:=\lim\limits_{d\rightarrow\infty}G_{x,d}=\{G_{\theta
	\lambda_{1}}(x),G_{\theta\lambda_{2}}(x),\dots\}$ and its supremum and 
	infimum are denoted as $\tilde{M}^{(1)}_x=\sup \tilde{G}^{(1)}_x$ and $ 
	\tilde{m}^{(1)}_x=\inf \tilde{G}^{(1)}_x$. For $k=2,3,\dots$, we denote $\tilde{G}^{(k)}_x = \tilde{G}^{(k-1)}_x \setminus \{\tilde{M}
	^{(k-1)}_x\}$, $\tilde{M}^{(k)}_x = \sup G^{(k)}_x.$
	
	Despite the structural similarities between comonotonic shock and common shock, the asymptotic analysis of $d$ in comonotonic shock is intricate. Unlike the common shock with only one shock variable $Y_{0}$, the number of shock variables in comonotonic shock tends to infinity as $d\rightarrow \infty $. This expansion poses significant challenges to the asymptotic analysis. To address this problem, we introduce the following assumptions.
	
	\begin{assumption}
	\label{as an comonotonic assumption} There exists $k\in\mathbb{N}_+$, such 
	that $G^{(1)}\supsetneqq G^{(2)}\supsetneqq\dots\supsetneqq G^{(k)} = 
	G^{(k+1)}$. 
	\end{assumption}
	
	\begin{assumption}
	\label{as an comonotonic assumption2} 
	There exists $k\in\mathbb{N}_+$, such 
	that $\tilde{G}^{(1)}_x\supsetneqq \tilde{G}^{(2)}_x\supsetneqq\dots\supsetneqq
	\tilde{G}^{(k)}_x = \tilde{G}^{(k+1)}_x$ for a given $x$.
	\end{assumption}
	
	\begin{proposition}
	\label{as an comonotonic prop4} Let $\bm{X}=(X_1,\dots,X_d)$ follows the 
	multivariate comonotonic shock Poisson distribution defined in Definition  %
	\ref{def-comonotonic shock}. Then if Assumption \ref{as an comonotonic
		assumption} holds, we have  
	\begin{equation*}
		F_{M^{\bm{X}}}(x)\sim c F_{M^{\bm{Y}}}(x), \quad \text{as }
		d\rightarrow\infty.
	\end{equation*}
	If Assumption \ref{as an comonotonic assumption2} holds, we have  
	\begin{equation*}
		F_{m^{\bm{X}}}(x)\sim \tilde{M}_x^{(k)} + \tilde{c}(x), \quad \text{as }
		d\rightarrow\infty,
	\end{equation*}
	where $c$ and $\tilde{c}(x)$ are given in (\ref{as an comonotonic eq6}) and
	( \ref{as an comonotonic eq7}), respectively. 
	\end{proposition}
	
	\begin{proof}
	Since in Theorem \ref{comonotonic th}, the number of summation operations in (\ref{comonotonic M eq1})-(\ref{comonotonic m eq2}) is the same as the dimension of the multivariate Poisson distribution, making it difficult to perform asymptotic analysis of $d$ directly using them. To this end, we derive $F_{M^{\bm{X}}}(x)$ and $F_{m^{\bm{X}}}(x)$ by conditioning on the underlying uniform random variable $U$ that generates the comonotonic shock random vector:  
	\begin{align}
		F_{M^{\bm{X}}}(x) =
		\int_{0}^{1}\prod_{j=1}^{d}G_{(1-\theta)\lambda_{j}}\left(x-G^{-1}_{\theta
			\lambda_{j}}(u)\right)du,   \quad
		F_{m^{\bm{X}}}(x) = \int_{0}^{1} \left(1-\prod_{j=1}^{d}\bar{G}
		_{(1-\theta)\lambda_{j}}\left(x-G^{-1}_{\theta\lambda_{j}}(u)\right)
		\right)du.  \label{as an comonotonic eq4}
	\end{align}
	Notably, although the explicitness of (\ref{as an comonotonic eq4}) is lower than that of (\ref{comonotonic M eq1})-(\ref{comonotonic m eq2}), the former are more suitable for the asymptotic analysis of $d$.
	
	From Assumption \ref{as an comonotonic assumption}, let us rearrange the 
	elements in $G^{(1)}$, with the first $k-1$ elements from  
	$G_{\theta\lambda_{1}}(0)=m^{(1)}$ to $G_{\theta\lambda_{k-1}}(0)=m^{(k-1)}$
	in sequence. Notice that for $u\in(m^{(i)},m^{(i+1)})$, $i=0,1,\dots,k-1$, we have,  $G_{\theta\lambda_{l}}^{-1}(u)\ge 1,l=1,2\dots,i,$ and $G_{\theta\lambda_{l}}^{-1}(u)=0,l=i+1$.
	Hence, we have,  
	\begin{equation}  \label{as an comonotonic eq5}
		\begin{split}
			\lim\limits_{d \to \infty}\frac{F_{M^{\bm{X}}}(x)}{F_{M^{\bm{Y}}}(x)}
			&=\int_{0}^{1}\prod_{j=1}^{\infty} \frac{G_{(1-\theta)\lambda_{j}}
				\left(x-G^{-1}_{\theta\lambda_{j}}(u)\right)}{G_{(1-\theta)\lambda_{j}}
				\left(x\right)} du =c + \int_{m^{(k)}}^{1}\prod_{j=1}^{\infty} \frac{
				G_{(1-\theta)\lambda_{j}}\left(x-G^{-1}_{\theta\lambda_{j}}(u)\right)}{
				G_{(1-\theta)\lambda_{j}}\left(x\right)} du,
		\end{split}%
	\end{equation}
	where  
	\begin{equation}  \label{as an comonotonic eq6}
		\begin{split}
			c =\int_{m^{(0)}}^{m^{(1)}}1du + \int_{m^{(1)}}^{m^{(2)}}\frac{
				G_{(1-\theta)\lambda_{1}}(x-G_{\theta\lambda_{1}^{-1}}(u))}{
				G_{(1-\theta)\lambda_{1}}(x)}du + \dots +
			\int_{m^{(k-1)}}^{m^{(k)}}\prod_{j=1}^{k-1} \frac{G_{(1-\theta)\lambda_{j}}
				\left(x-G^{-1}_{\theta\lambda_{j}}(u)\right)}{G_{(1-\theta)\lambda_{j}}
				\left(x\right)} du
		\end{split}%
	\end{equation}
	is a constant.
	
	By the supremum property, for $u\in(m^{(k)},1)$, there exists $
	u_1:=G_{\theta\lambda_{a_{1}}}(0)\in G^{(k)}\subset G^{(1)}$, such that $
	G_{\theta\lambda_{a_{1}}}(0) \in(m^{(k)},u)$. Similarly, there exists $
	u_2:=G_{\theta\lambda_{a_{2}}}(0)\in G^{(k)}\subset G^{(1)}$, such that $
	G_{\theta\lambda_{a_{2}}}(0) \in(m^{(k)},u_1)$. Repeat this operation, we 
	can find a sequence $\{a_n\}$ such that no two of $\{u_n\}=\{G_{\theta
		\lambda_{a_{n}}}\}$ are same and $G_{\theta\lambda_{a_n}}^{-1}(u)\ge1$ for $
	n=1,\dots$. Thus, we have,  
	\begin{align*}
	&\int_{m^{(k)}}^{1}\prod_{j=1}^{\infty} \frac{G_{(1-\theta)\lambda_{j}}
		\left(x-G^{-1}_{\theta\lambda_{j}}(u)\right)}{G_{(1-\theta)\lambda_{j}}
		\left(x\right)} du \\
		&= \int_{m^{(k)}}^{1}\prod_{n=1}^{\infty} \frac{G_{(1-\theta)
			\lambda_{a_{n}}}\left(x-G^{-1}_{\theta\lambda_{a_{n}}}(u)\right)}{
		G_{(1-\theta)\lambda_{a_{n}}}\left(x\right)}\prod_{\text{Others }
		\lambda_{j}\in\Lambda}\frac{G_{(1-\theta)\lambda_{j}}\left(x-G^{-1}_{\theta
			\lambda_{j}}(u)\right)}{G_{(1-\theta)\lambda_{j}}\left(x\right)} du = 0.
	\end{align*}
	
	Similarly, for $F_{m^{\bm{X}}}(x)$, we have,  
	\begin{align*}
		\lim\limits_{d \to \infty}&\int_{0}^{1} \left(1-\prod_{j=1}^{d}\bar{G}
		_{(1-\theta)\lambda_{j}}\left(x-G^{-1}_{\theta\lambda_{j}}(u)\right)\right)du
		\\
		&=\int_{0}^{\tilde{m}_x^{(1)}} 1 du + \int_{\tilde{m}_x^{(1)}}^{\tilde{M}
			_x^{(1)}} \left(1-\prod_{j=1}^{\infty}\bar{G}_{(1-\theta)\lambda_{j}}
		\left(x-G^{-1}_{\theta\lambda_{j}}(u)\right)\right)du + \int_{\tilde{M}
			_x^{(1)}}^{1} \left(1-\prod_{j=1}^{\infty}1 \right)du \\
		&=\tilde{m}_x^{(1)}+ \int_{\tilde{m}_x^{(1)}}^{\tilde{M}_x^{(1)}}
		\left(1-\prod_{j=1}^{\infty}\bar{G}_{(1-\theta)\lambda_{j}}\left(x-G^{-1}_{
			\theta\lambda_{j}}(u)\right)\right)du.
	\end{align*}
	
	From Assumption \ref{as an comonotonic assumption2}, let us rearrange the 
	elements in $\tilde{G}^{(1)}_x$, with the first $k-1$ elements from $\tilde{%
		M		}^{(1)}_x:=G_{\theta\lambda_{1}}(x)$ to $\tilde{M}^{(k-1)}_x:=G_{\theta
		\lambda_{k-1}}(x)$ in sequence. Similar to (\ref{as an comonotonic eq5}), we
	have,  
	\begin{equation*}
		\begin{split}
			\int_{\tilde{m}_x^{(1)}}^{\tilde{M}_x^{(1)}} \left(1-\prod_{j=1}^{\infty} 
			\bar{G}_{(1-\theta)\lambda_{j}}\left(x-G^{-1}_{\theta\lambda_{j}}(u)\right)
			\right)du =\int_{\tilde{m}_x^{(1)}}^{\tilde{M}_x^{(k)}}
			\left(1-\prod_{j=1}^{\infty}\bar{G}_{(1-\theta)\lambda_{j}}\left(x-G^{-1}_{
				\theta\lambda_{j}}(u)\right)\right)du + \tilde{c}(x),
		\end{split}%
	\end{equation*}
	where $\tilde{c}(x)=0$ for $k=1$, and when $k=1, 2,\dots,$, $\tilde{c}(x)$ is given by
	\begin{equation}  \label{as an comonotonic eq7}
		\begin{split}
			\tilde{c}(x)=\int_{\tilde{M}_x^{(k)}}^{\tilde{M}_x^{(k-1)}}
			\left(1-\prod_{j=1}^{k-1}\bar{G}_{(1-\theta)\lambda_{j}}\left(x-G^{-1}_{
				\theta\lambda_{j}}(u)\right)\right)du+ \dots + \int_{\tilde{M}_x^{(2)}}^{ 
				\tilde{M}_x^{(1)}}\left(1-\bar{G}_{(1-\theta)\lambda_{1}}\left(x-G^{-1}_{
				\theta\lambda_{1}}(u)\right)\right) du
		\end{split}%
	\end{equation}
	is a constant given $x$.
	
	By the supremum property, we can find a sequence $\{b_n\}$ such that $
	G_{\theta\lambda_{b_n}}^{-1}(u)\le x$ for $u\in(\tilde{m}_x^{(1)},\tilde{M}
	_x^{(k)})$ and $n=1,\dots$. Thus, we have,  
	\begin{align*}
		&\int_{\tilde{m}_x^{(1)}}^{\tilde{M}_x^{(k)}} \left(1-\prod_{j=1}^{\infty} 
		\bar{G}_{(1-\theta)\lambda_{j}}\left(x-G^{-1}_{\theta\lambda_{j}}(u)\right)
		\right)du \\
		&=\int_{\tilde{m}_x^{(1)}}^{\tilde{M}_x^{(k)}} \left( 1-\prod_{n=1}^{\infty} 
		\bar{G}_{(1-\theta)\lambda_{b_{n}}}\left(x-G^{-1}_{\theta\lambda_{b_{n}}}(u)
		\right) \prod_{\text{Others }\lambda_{j}\in\Lambda}\bar{G}
		_{(1-\theta)\lambda_{j}}\left(x-G^{-1}_{\theta\lambda_{j}}(u)\right) \right)
		du= \tilde{M}_x^{(k)} - \tilde{m}_x^{(1)}.
	\end{align*}
	\end{proof}
	
	\begin{remark}
	\label{remark 4.2} If $k=1$ satisfies Assumption \ref{as an comonotonic
		assumption}, we observe $F_{M^{\bm{X}}}(x)\sim m^{(1)}F_{M^{\bm{Y}}}(x)$ as $%
	d\rightarrow \infty $. Notably, $m^{(1)}$ represents the infimum of the CDF
	values of each comonotonic shock variable, $Z_{i}$, at $x=0$. This result is
	similar to that under the common shock model (Proposition \ref{as an common
		prop3}). In both dependent structures, the CDFs of $M^{\bm{X}}$ are
	asymptotically equivalent to $F_{\bm{Y}}(x)$ multiplied by a constant
	related to the value of the CDF of the \textquotedblleft shock
	variables\textquotedblright\ at $0$. Similarly, if $k=1$ satisfies
	Assumption \ref{as an comonotonic assumption2}, we have $F_{m^{\bm{X}%
	}}(x)\sim \tilde{M}_{x}^{(1)}$ as $d\rightarrow \infty $, also closely
	resembling the results under the common shock model.
	
	\end{remark}

	\Cref{remark 4.2} illustrates the differences and connections between the
	structurally similar common shock and comonotonic shock structures, 
	highlighting that while the extension of comonotonic shock is 
	straightforward, it leads to significantly different probabilistic outcomes,
	particularly in the asymptotic analysis of $d$ on $M^{\bm{X}}$ and $m^{\bm{X}
	}$. 
	The difference becomes more evident when the dimension is high.
	
	\subsection{Thinning-dependence structure}
	
	In this subsection, we discuss the asymptotic analysis for parameters $
	p_i^k, $ for $i=1,\dots,d$, $k=1,\dots,l$ for the multivariate 
	thinning-dependence Poisson distribution as presented in Proposition \ref{as
	an thin prop1}, and \ref{as an thin prop2}. Meanwhile, we also consider the 
	asymptotic result on the dimension of $\bm{X}$ in Proposition \ref{as an
	thin prop3}.
	
	\begin{proposition}
	\label{as an thin prop1} Let $\bm{X}=(X_1,\dots,X_d)$ follows the 
	multivariate thinning-dependence Poisson distribution defined in Definition  %
	\ref{def-thinning-dependence}, then  
	\begin{align*}
		F_{M^{\bm{X}}}\left(x\right) \sim F_{X_i}(x), \quad
		F_{m^{\bm{X}}}(x) \sim F_{m_{-i}^{\bm{X}}}(x), ~\text{as}~p_i^k\rightarrow1,k=1,\dots,l.
	\end{align*}
	\end{proposition}
	
	\begin{proof}
	For any $i=1,\dots,d$, $k=1,\dots,l$, and $x_i^k=0,1,\dots,y_k$, we have,  
	\begin{align}  \label{as an thinnin eq1}
		\lim\limits_{p_i^k \to 1}\mathbb{P}(X_i^k=x_i^k|Y_k=y_k)= 
		\begin{cases}
			\begin{aligned} &1, \quad&&\text{if}~x_i^k=y_k, \\ &0,
				\quad&&\text{if}~x_i^k<y_k. \end{aligned}%
		\end{cases}%
	\end{align}
	Recall Theorem \ref{thinning th}. For both $F_{M^{\bm{X}}}(x)$ and $F_{m^{ %
			\bm{X}}}(x)$, $p_i^k$, $k=1,\dots,l$ appear only in $\sum_{s=0}^{x}h_j
	\left(s,y_1,\dots,y_l\right)$. This quantity can be expressed as: 
	\begin{align}
		&\sum_{s=0}^{x}h_j\left(s,y_1,\dots,y_l\right) = \mathbb{P}
		\left(\sum_{k=1}^{l}X_i^k \le x \vert Y_1=y_1,\dots,Y_l=y_l\right) \notag\\
		&=\sum_{x_i^1+\dots+x_i^l\le x}\mathbb{P}\left(X_i^1=x_i^1,\dots,X_i^l=x_i^l
		\vert Y_1=y_1,\dots,Y_l=y_l\right) \label{as an thinnin eq2} \\
		&=\sum_{x_i^1+\dots+x_i^l\le x}\mathbb{P}\left(X_i^1=x_i^1\vert
		Y_1=y_1\right) \dots\mathbb{P}\left(X_i^l=x_i^l\vert Y_l=y_l\right),\notag
	\end{align} 
	where the last equality is from Lemma \ref{lemma-3.2}. Thus, from (\ref{as
		an thinnin eq1}), as $p_i^k \to 1,k=1,\dots,l$, (\ref{as an thinnin eq2}) 
	can be rewritten as follow:  
	\begin{equation*}
		\begin{split}
			\lim\limits_{p_i^k \to 1,k=1,\dots,l} \sum_{x_i^1+\dots+x_i^l\le x}\mathbb{%
				P }\left(X_i^1=x_i^1\vert Y_1=y_1\right) \dots\mathbb{P}\left(X_i^l=x_i^l%
			\vert Y_l=y_l\right) = 
			\begin{cases}
				& 1, \quad \text{if}~x_i^1=y_1, \dots, x_i^l=y_l, \\ 
				& 0, \quad \text{otherwise}.%
			\end{cases}%
		\end{split}%
	\end{equation*}
	Therefore, we have,
	\begin{align*}
		\lim\limits_{p_i^k \to 1,k=1,\dots,l}\frac{F_{M^{\bm{X}}}\left(x\right)}{
			F_{X_i}(x)}=\frac{\sum_{y_1+\cdots+y_l \le x}^{}\prod_{k=1}^{l}\frac{
				\theta_k^{y_k}e^{-\theta_k}}{\Gamma\left(y_k+1\right)}}{\sum_{y_1+\cdots+y_l
				\le x}^{}\prod_{k=1}^{l}\frac{\theta_k^{y_k}e^{-\theta_k}}{
				\Gamma\left(y_k+1\right)}} =1.
	\end{align*}
	The asymptotic result of $F_{m^{\bm{X}}}(x)$ can be obtained in a similar
	manner. 
		\end{proof}
		
		\begin{proposition}
	\label{as an thin prop2} Let $\bm{X}=(X_1,\dots,X_d)$ follows the 
	multivariate thinning-dependence Poisson distribution defined in Definition  %
	\ref{def-thinning-dependence}, then  
	\begin{align*}
		F_{M^{\bm{X}}}\left(x\right) \sim F_{M_{-i}^{\bm{X}}}\left(x\right), \quad
		F_{m^{\bm{X}}}(x) \sim 1, 
		~\text{as}~ p_i^k\rightarrow0,~k=1,\dots,l.
	\end{align*}
	\end{proposition}
	
	\begin{proof}
	For any $i=1,\dots,d$, $k=1,\dots,l$, and $x_i^k=0,1,\dots,y_k$, we have,  
	\begin{align*}
		\lim\limits_{p_i^k \to 0}\mathbb{P}(X_i^k=x_i^k|Y_k=y_k)= 
		\begin{cases}
			\begin{aligned} &1, \quad&&\text{if}~x_i^k=0, \\ &0,
				\quad&&\text{if}~x_i^k=1,\dots,y_k. \end{aligned}%
		\end{cases}%
	\end{align*}
	The remaining proof is in the same line with Proposition \ref{as an thin prop1}. 
	\end{proof}
	
	From \Cref{def-thinning-dependence}, we can directly deduce that $M^{\bm{X}}\leq Y_{1}+\dots +Y_{l}$, and the lower bound of $m^{\bm{X}}\geq 0$; see also Remark \ref{remark2.2-1}. When there are more marginal risks, i.e. the dimension $d$ is high, more trials are involved in the system, and it is intuitive to expect the upper/lower bound is reached with a higher chance. Consequently, we prove the following proposition of the asymptotic results for $M^{\bm{X}}$ and $m^{\bm{X}}$ when $d\rightarrow \infty $, which is indeed consistent with intuitions.
	\begin{proposition}
	\label{as an thin prop3} Let $\bm{X}=(X_1,\dots,X_d)$ follows the 
	multivariate thinning-dependence Poisson distribution defined in Definition  %
	\ref{def-thinning-dependence}, then  
	\begin{align*}
		F_{M^{\bm{X}}}\left(x\right) \sim F_{Y_1}*F_{Y_2}*\cdots *F_{Y_{l}}(x),
		\quad F_{m^{\bm{X}}}(x) \sim1,~\text{as}~d\rightarrow\infty.
	\end{align*}
	where $*$ denotes the convolution operator.
	
	\end{proposition}
	
	\begin{proof}
	Notably, for $j=1,\dots,d$, $\mathbb{P}\left(X_j \le x \vert 
	Y_1=y_1,\dots,Y_l=y_l\right)$ and $1-\mathbb{P}\left(X_j \le x \vert 
	Y_1=y_1,\dots,Y_l=y_l\right)$ are strictly positive and less than $1$ when $
	y_1+\cdots+y_l\ge x+1$. Thus, in this scenario,  
	\begin{equation*}
		\lim\limits_{d \to \infty}\prod_{j=1}^{d}\mathbb{P}\left(X_j \le x \vert
		Y_1=y_1,\dots,Y_l=y_l\right)=0, \quad \lim\limits_{d \to
			\infty}\prod_{j=1}^{d}1-\mathbb{P}\left(X_j \le x \vert
		Y_1=y_1,\dots,Y_l=y_l\right)=0.
	\end{equation*}
	Therefore, when $d\rightarrow\infty$, (\ref{thinning cdf M}) and (\ref%
	{thinning cdf m}) can be written as:  
	\begin{equation*}
		\lim\limits_{d \to \infty}F_{M^{\bm{X}}}\left(x\right) =\sum_{y_1+\dots+y_l
			\le x}^{}\prod_{k=1}^{l}\frac{\theta_k^{y_k}e^{-\theta_k}}{
			\Gamma\left(y_k+1\right)} = F_{Y_1+\cdots+Y_l}(x),
	\end{equation*}
	and,  
	\begin{align*}
		\lim\limits_{d\to\infty}F_{m^{\bm{X}}}(x) =\sum_{y_1+\cdots+y_l \le
			x}^{}\prod_{k=1}^{l}\mathbb{P}\left(Y_k=y_k\right)+\sum_{y_1+\cdots+y_l \ge
			x+1}^{}\prod_{k=1}^{l}\mathbb{P}\left(Y_k=y_k\right) =\prod_{k=1}^{l}\sum_{y_k=0}^{\infty}\mathbb{P}\left(Y_k=y_k\right) =1.
	\end{align*}
	\end{proof}
	
	\begin{remark}
	\label{remark 4.4}  The asymptotic results for $d$ under the thinning-dependence structure is significantly different, compared to those under the common shock and comonotonic shock structures. Specifically, as $d\rightarrow \infty $, $M^{\bm{X}}$'s behavior is determined by the common background factors under the thinning-dependence structure, whereas under the common shock and comonotonic shock structures, $M^{\bm{X}}$ behaves consistently with the idiosyncratic factors. 
	%
	%
	\end{remark}
	
	Combining \Cref{remark 4.2,remark 4.4}, we find that the
	differences in the asymptotic analysis of $d$ under the three structures
	correspond precisely to the intrinsic differences in how these structures
	are constructed, particularly in the way the common factor is introduced.
	This not only sheds new light on analyzing dependent structures but also
	plausibly suggests that similar results may be obtained when utilizing these structures to construct other distributions.
	
	\section{Numerical results} \label{sec5}
	In this section, we present numerical results of the CDFs of the 
	multivariate Poisson distribution and its asymptotic analysis under the 
	common shock, comonotonic shock, and thinning-dependence structures. Due to space limitation, the numerical examples focus solely on the distribution of $M^{\bm{X}}$.
	
	\subsection{Numerical result of CDFs}
	
	To compare the effects of different dependent structures on the maximum (minimum) of multivariate Poisson distributions, we ensure that the marginal distributions of the three multivariate Poisson distributions are identical in our numerical example, and that the correlations within each distribution are also consistent. Intuitively, the most direct way to ensure the same correlation is to equalize the correlation coefficients of the three multivariate Poisson distributions. However, this is often difficult to achieve while maintaining equal marginal distributions. Therefore, we use the average correlation coefficient to characterize the internal correlation within each multivariate distribution in this paper. To illustrate the results, we take the 3-dimensional case as an example. Specifically, for a 3-dimensional random vector $\bm{X}=\left(X_1,X_2,X_3\right)$, the average correlation coefficient $\bar{\rho}$ is given by  
	\begin{equation*}
	\bar{\rho} = \frac{\text{Var}(X_1+X_2+X_3)-\text{Var}(X_1)-\text{Var}(X_2)- 
		\text{Var}(X_3)}{2\left( \sqrt{\text{Var}(X_1)\text{Var}(X_2)} + \sqrt{\text{
				Var}(X_2)\text{Var}(X_3)} + \sqrt{\text{Var}(X_1)\text{Var}(X_3)} \right)}.
				\end{equation*}
				
				Hence, the average correlation coefficient of the multivariate common shock, comonotonic shock, thinning-dependence Poisson distribution, denoted as $\rho_1$, $\rho_2$ and $\rho_3$, are given by  
				\begin{align}
	\bar{\rho}_1&=\frac{3\theta_0}{\sqrt{\left(\theta_0+\theta_1\right)\left(
			\theta_0+\theta_2\right)}+\sqrt{\left(\theta_0+\theta_1\right)\left(
			\theta_0+\theta_3\right)}+\sqrt{\left(\theta_0+\theta_2\right)\left(
			\theta_0+\theta_3\right)}},  \notag \\
	\bar{\rho}_2&=\frac{m_{\lambda_1,\lambda_2}\left(\theta\right)+m_{\lambda_1,
			\lambda_3}\left(\theta\right)+m_{\lambda_2,\lambda_3}\left(\theta\right)}{ 
		\sqrt{\lambda_1 \lambda_2}+\sqrt{\lambda_1 \lambda_3}+\sqrt{\lambda_2
			\lambda_3}},  \label{average corr} \\
	\bar{\rho}_3&=\frac{\sum_{k=1}^{l}p_1^k p_2^k+\sum_{k=1}^{l}p_1^k
		p_3^k+\sum_{k=1}^{l}p_2^k p_3^k}{\sqrt{\sum_{k=1}^{l}p_1^k
			\sum_{k=1}^{l}p_2^k}+\sqrt{\sum_{k=1}^{l}p_1^k \sum_{k=1}^{l}p_3^k}+\sqrt{
			\sum_{k=1}^{l}p_2^k \sum_{k=1}^{l}p_3^k}},  \notag
			\end{align}
			where $m_{\lambda_j,\lambda_k}\left(\theta\right)=\text{cov}
			\left(X_j,X_k\right)=\sum_{m=0}^{\infty}\sum_{n=0}^{\infty}\min\left\{ 
			\overline{G}_{\theta \lambda_j}\left(m\right),\overline{G}_{\theta
	\lambda_k}\left(m\right)\right\}-\theta^2 \lambda_j \lambda_k$.
	
	In \Cref{CDF results}, we compare the CDFs of $M^{\bm{X}}$ for three types 
	of multivariate Poisson. In the case of $l=1$, for the multivariate common 
	shock Poisson distribution, we set $\theta _{0}=3.3804,\theta 
	_{1}=2.6196,\theta _{2}=3.6196,\theta _{3}=4.6196$. For the multivariate 
	common shock Poisson distribution, we set $\theta _{0}=3.3804,\theta 
	_{1}=2.6196,\theta _{2}=3.6196,\theta _{3}=4.6196$; for the multivariate 
	comonotonic shock Poisson distribution, we set $\lambda _{1}=6,\lambda 
	_{2}=7,\lambda _{3}=8,\theta =0.5$; for the multivariate thinning-dependence
	Poisson distribution, we set $l=1,\theta 
	_{1}=15.3829,p_{1}=0.39,p_{2}=0.4551,p_{3}=0.5201$. This setting ensures 
	equal marginal distributions and similar average correlation coefficients, 
	with $\bar{\rho}_{1}\approx \bar{\rho}_{2}\approx \bar{\rho}_{3}\approx 0.48$
	. In the case of $l=2$, for the multivariate common shock Poisson 
	distribution, we set $\theta _{0}=8.2,\theta _{1}=\theta _{2}=0.8,\theta 
	_{3}=1.8$; for the multivariate comonotonic shock Poisson distribution, we 
	set $\lambda _{1}=\lambda _{2}=9,\lambda _{3}=10,\theta =0.5$; For the 
	multivariate thinning-dependence Poisson distribution, we set $\theta 
	_{1}=8.2,\theta 
	_{2}=19.69,p_{1}^{1}=p_{2}^{1}=p_{3}^{1}=1,p_{1}^{2}=p_{2}^{2}=0.0406,p_{3}^{2}=0.0914  
	$. This setup also ensures equal marginal distributions and similar average 
	correlation coefficients, with $\bar{\rho}_{1}\approx \bar{\rho}_{2}\approx  
	\bar{\rho}_{3}\approx 0.9$.
	
	\begin{table}[h!]
	\footnotesize
	\caption{Numerical results of the multivariate Poisson distribution.}
	\label{CDF results}
	\centering
	\setlength{\abovecaptionskip}{0.1cm} \setlength{\tabcolsep}{9pt}  %
	\renewcommand{\arraystretch}{1.4}  
	\begin{tabular}{ccccccc}
		\hline
		\multirow{2}{*}{\textbf{x}} & \multicolumn{3}{c}{\textbf{l=2}} & 
		\multicolumn{3}{c}{\textbf{l=1}} \\ \cline{2-7}
		& \textbf{Common} & \textbf{Comonotonic} & \textbf{Thinning} & \textbf{%
			Common } & \textbf{Comonotonic} & \textbf{Thinning} \\ \hline
		1 & 1.58$\times 10^{-5}$ & 1.50$\times 10^{-5}$ & 1.77$\times 10^{-5}$ & 
		6.70 $\times 10^{-7}$ & 8.89$\times 10^{-7}$ & 7.17$\times 10^{-5}$ \\ 
		2 & 1.17$\times 10^{-3}$ & 1.17$\times 10^{-3}$ & 1.27$\times 10^{-3}$ & 
		3.01 $\times 10^{-5}$ & 3.85$\times 10^{-5}$ & 7.71$\times 10^{-4}$ \\ 
		3 & 5.37$\times 10^{-3}$ & 5.48$\times 10^{-3}$ & 5.68$\times 10^{-3}$ & 
		4.57 $\times 10^{-4}$ & 5.65$\times 10^{-4}$ & 4.41$\times 10^{-3}$ \\ 
		4 & 1.77$\times 10^{-2}$ & 1.83$\times 10^{-2}$ & 1.85$\times 10^{-2}$ & 
		3.42 $\times 10^{-3}$ & 4.08$\times 10^{-3}$ & 1.66$\times 10^{-2}$ \\ 
		5 & 4.57$\times 10^{-2}$ & 4.72$\times 10^{-2}$ & 4.72$\times 10^{-2}$ & 
		1.55 $\times 10^{-2}$ & 1.78$\times 10^{-2}$ & 4.66$\times 10^{-2}$ \\ 
		6 & 0.0973 & 0.1004 & 0.0997 & 0.0488 & 0.0540 & 0.1039 \\ 
		7 & 0.1771 & 0.1821 & 0.1802 & 0.1162 & 0.1247 & 0.1937 \\ 
		8 & 0.2833 & 0.2903 & 0.2870 & 0.2229 & 0.2336 & 0.3123 \\ 
		9 & 0.4076 & 0.4158 & 0.4114 & 0.3610 & 0.3718 & 0.4480 \\ 
		10 & 0.5371 & 0.5456 & 0.5406 & 0.5119 & 0.5208 & 0.5846 \\ 
		11 & 0.6587 & 0.6662 & 0.6616 & 0.6542 & 0.6604 & 0.7073 \\ 
		12 & 0.7626 & 0.7683 & 0.7649 & 0.7724 & 0.7762 & 0.8069 \\ 
		13 & 0.8442 & 0.8481 & 0.8458 & 0.8605 & 0.8625 & 0.8805 \\ 
		14 & 0.9034 & 0.9058 & 0.9045 & 0.9200 & 0.9209 & 0.9305 \\ 
		15 & 0.9433 & 0.9447 & 0.9440 & 0.9569 & 0.9573 & 0.9620 \\ \hline
	\end{tabular}%
	\end{table}

	According to (\ref{thinning def eq}), when $l=1$, each $X_{i}$ in the 
	thinning-dependence structure is derived from only one $X_{i}^{1}$, meaning 
	that all $X_{i}$ are only dependent on the same $Y_{1}$. This points out a 
	significant difference between the thinning-dependence structure and the 
	other two structures. However, when $l=2$, not only all $X_{i}$ have two 
	cumulative sources, but also, if $p_{i}^{1}=1$ simultaneously, the first 
	cumulative source $X_{i}^{1}$ of $X_{i}$ becomes a common $Y_{1}$. In this 
	case, the thinning-dependence structure closely resembles the other two 
	structures. The numerical results in \Cref{CDF results} also confirm this 
	point.  When $l=1$, the CDF of the thinning-dependence structure is greater
	than that of the common shock and comonotonic shock structures, indicating
	that the $M^{\bm{X}}$ in the former structure, which adds the common factors
	indirectly, is indeed larger than in the common shock and comonotonic shock
	structures, which directly add common factors. Moreover, the gap in CDF
	values between the thinning-dependence structure and the common
	(comonotonic) shock structures is significantly larger when $l=2$,
	especially for smaller values of $x$. These findings highlight the
	significant impact of different dependence structures on the maximum of the
	multivariate Poisson distribution.
	
	
	\subsection{Numerical result of asymptotic analysis}
	
	In this subsection, we first provide numerical examples for the dimension $d$
	: For the multivariate common shock Poisson distribution, we set~$
	\theta_0=5, \theta_1=\cdots=\theta_d=5$; For the multivariate comonotonic 
	shock Poisson distribution, we set $\lambda_j=8-2/j, j=1,2,\dots,d, 
	\theta=0.5$; For the multivariate thinning-dependence Poisson distribution, 
	we set $l=1, \theta_1=8, p_j=0.7, j=1,2,\dots,d$. The asymptotic analysis 
	results are presented in \Cref{as an d table}. Secondly, we provide 
	numerical examples for the asymptotic analysis of other parameters. For the 
	multivariate common shock Poisson distribution, in \Cref{as an common prop1}
	, we set $d=10$, $\theta_j=6$ for $j=1, \dots, 10$ and $j \ne i$; in  %
	\Cref{as an common prop2}, we set $d=2$, $\theta_1=\theta_2=2$. For the 
	multivariate comonotonic shock Poisson distribution, in  
	\Cref{as an
	comonotonic prop1}, we set $d=3$, $\theta=1/2$, $\lambda_1=\lambda_2=20$;
	in  \Cref{as an comonotonic prop2}, we set $d=3$, $\lambda_1=\lambda_2=
	\lambda_3=10$; in \Cref{as an comonotonic prop3}, we set $d=3$, $
	\lambda_1=\lambda_2=\lambda_3=5$. For the multivariate thinning-dependence 
	Poisson distribution, in \Cref{as an thin prop1}, we set $l=1$, $d=10$, $
	\theta_1=10$, $p_j^1=0.5$ for $j=1, \dots, 10$ and $j \ne i$; in  
	\Cref{as an
	thin prop2}, we set $l=1$, $d=10$, $\theta_1=10$, $p_j^1=0.5$ for $j=1, 
	\dots, 10$ and $j \ne i$. The results are presented in \Cref{as an table2}. 
	The findings in both \Cref{as an d table} and \Cref{as an table2} indicate 
	that our obtained CDFs converge well to the asymptotic results.
	
	\begin{table}[h!]
	\footnotesize
	\caption{Numerical results for the asymptotic analysis of dimension $d$,
		illustrating the ratio of the asymptotic results (AR) to its corresponding
		maximum distribution (FM) when $x=7$.}
	\label{as an d table}
	\centering
	\setlength{\abovecaptionskip}{0.1cm} \setlength{\tabcolsep}{8pt}  %
	\renewcommand{\arraystretch}{1.4}  
	\begin{tabular}{llccccc}
		\hline
		\textbf{Dependent Structure} & \textbf{Proposition} & \textbf{d = 10} & 
		\textbf{d = 20} & \textbf{d = 50} & \textbf{d = 70} & \textbf{d = 200} \\ 
		\hline
		\multicolumn{1}{l|}{Common} & \Cref{as an common prop3} & 0.35464 & 0.71591
		& 0.99193 & 0.99938 & 1 \\ 
		\multicolumn{1}{l|}{Comonotonic} & \Cref{as an comonotonic prop4} & 0.15952
		& 0.35401 & 0.80749 & 0.92822 & 0.99389 \\ 
		\multicolumn{1}{l|}{Thinning} & \Cref{as an thin prop3} & 0.83133 & 0.91117
		& 0.98441 & 0.99520 & 1 \\ \hline
	\end{tabular}%
	\end{table}
	
	\begin{table}[h!]
	\footnotesize
	\caption{Numerical results for the asymptotic analysis of other parameters
		are presented, illustrating the ratio between the FM and its AR when $x=7$.}
	\label{as an table2}
	\centering
	\setlength{\abovecaptionskip}{0.1cm} \setlength{\tabcolsep}{6pt}  %
	\renewcommand{\arraystretch}{1.4} \centering
	\begin{tabular}{c|ccccccc}
		\hline
		\multirow{4}{*}{Common} & \multirow{2}{*}{\Cref{as an common prop1}} & $%
		\theta_i$ & 10 & 100 & 200 & 1000 & 5000 \\ 
		&  & FM/AR & 0.08830 & 0.71232 & 0.84105 & 0.96542 & 0.99296 \\ \cline{2-8}
		& \multirow{2}{*}{\Cref{as an common prop2}} & $\theta_0$ & 10 & 50 & 500 & 
		5000 & 8000 \\ 
		&  & AR/FM & 0.07647 & 0.42064 & 0.89728 & 0.98890 & 0.99304 \\ \hline
		\multirow{6}{*}{Comonotonic} & \multirow{2}{*}{\Cref{as an comonotonic
				prop1}} & $\lambda_i$ & 10 & 20 & 40 & 50 & 100 \\ 
		&  & FM/AR & 0.00229 & 0.06308 & 0.91665 & 1 & 1 \\ \cline{2-8}
		& \multirow{2}{*}{\Cref{as an comonotonic prop2}} & $\theta$ & 0.66 & 0.46 & 
		0.31 & 0.11 & 0.01 \\ 
		&  & FM/AR & 0.08487 & 0.09044 & 0.14666 & 0.46898 & 0.99210 \\ \cline{2-8}
		& \multirow{2}{*}{\Cref{as an comonotonic prop3}} & $\theta$ & 0.0005 & 
		0.3005 & 0.5005 & 0.9005 & 0.99999 \\ 
		&  & FM/AR & 0.65094 & 0.69323 & 0.72714 & 0.88655 & 0.99998 \\ \hline
		\multirow{4}{*}{Thinning} & \multirow{2}{*}{\Cref{as an thin prop1}} & $%
		p_i^k $ & 0.009 & 0.309 & 0.509 & 0.709 & 0.959 \\ 
		&  & FM/AR & 0.56523 & 0.57256 & 0.64028 & 0.80086 & 0.99029 \\ \cline{2-8}
		& \multirow{2}{*}{\Cref{as an thin prop2}} & $p_i^k$ & 0.951 & 0.751 & 0.551
		& 0.251 & 0.001 \\ 
		&  & FM/AR & 0.46755 & 0.77821 & 0.95398 & 0.99972 & 1 \\ \hline
	\end{tabular}%
	\end{table}
	
	\section{Discussion and conclusion}\label{sec6}
	
	In this paper, we comprehensively investigate the distribution of the
	maximum and minimum of multivariate Poisson distributions, which is the
	cornerstone of the discrete distributions. To the best of our knowledge,
	there is very limited literature on the maximum distribution for discrete
	random vectors. Our work not only derives explicit maximum and minimum distribution functions for three types of multivariate Poisson distributions but also conducts asymptotic analyses for all derived CDFs. These results address a research gap in the literature, providing valuable insights into relevant studies on this topic. In particular, our asymptotic analysis may indicate a new way to study the dependence structures for other multivariate distributions.
	
	As a matter of fact, the Poisson distribution, being a classical discrete 
	distribution, has a wide range of applications, such as in credit risk \citep
	{liang2012reduced}, reinsurance \citep{yuen2015optimal}, and price modeling  
	\citep{holy2022modeling}. In future research, we also anticipate utilizing 
	our findings on the distribution of the maximum and minimum of multivariate 
	Poisson distributions to model such practical problems, providing new 
	perspectives and solutions to these important areas.

	
	\section*{Funding}
	This work is supported by the National Natural Science Foundation of China [grant number 12371474]; the Postgraduate Research \& Practice Innovation Program of Jiangsu Province [grant number KYCX24\underline{~}3353].

	\bibliographystyle{plainnat}
	\bibliography{max_ref}
\end{document}